\documentclass[11pt,a4paper]{amsart}
\usepackage{amssymb,amsmath}
\usepackage{graphicx, tikz,caption}
\usepackage[LGR,T1]{fontenc}
\usepackage{hyperref}
\numberwithin{equation}{section}

\newtheorem{theorem}{Theorem}[section]
\newtheorem{prop}[theorem]{Proposition}
\newtheorem{problem}[theorem]{Problem}

\newtheorem{lemma}[theorem]{Lemma}
\newtheorem{definition}[theorem]{Definition}
\newtheorem{assumption}[theorem]{Assumption}

\theoremstyle{remark}

\newtheorem{remark}[theorem]{Remark}
\newtheorem{example}[theorem]{Example}

\begin{document}

\title{Skew-product systems over infinite interval exchange transformations}
\author{Henk Bruin}
\address{Henk Bruin, Faculty of Mathematics, University of Vienna, Oskar-Morgenstern-Platz 1, 1090 Vienna, Austria}
\email {henk.bruin@univie.ac.at} 
\author{Olga Lukina}
\address{Olga Lukina, Mathematical Institute, Leiden University, P.O. Box 9512, 2300 RA Leiden, The Netherlands}
\email{o.lukina@math.leidenuniv.nl}
\date{Version August 26, 2024. Revised March 16, 2025}

\maketitle

\begin{abstract}
We study the ergodic properties (recurrence, discrepancy, diffusion coefficients and ergodicity itself) of
a class of $\mathbb{Z}$-extensions over infinite interval exchange transformations called rotated odometers.
The choice of a skew-function is motivated by the use in the study of parallel flows on a particular staircase manifold of
infinite genus.
\end{abstract}

\section{Introduction}\label{sec:intro}

A rotated odometer $F_\pi:I \to I$ is a type of infinite interval exchange transformation (IIET)
on the half-open unit interval $I = [0,1)$, introduced in \cite{BL23-JLMS}.
In this paper, we study $\mathbb{Z}$-extensions, that is, skew-products on $I \times \mathbb{Z}$ 
over these rotated odometers, and ask ourselves questions on their recurrence or transience, discrepancy, diffusion coefficients, and ergodicity.

\subsection{Rotated odometers}
The dyadic odometer is the group action on $\{0,1\}^{\mathbb N}$ given  by the addition of $1$ to the left-most digit with infinite carry to the right.
The dyadic odometer has a structure of a profinite group, and so admits a unique ergodic invariant measure
(Haar measure). It is well known that, as a measure-preserving dynamical system, the dyadic odometer is isomorphic to the Lebesgue-measure preserving von Neumann-Kakutani map on the unit interval $I = [0,1)$:
\begin{align}\label{eq-odometer}
 \mathfrak{a}(x) =  x - (1-3 \cdot 2^{-n}) \qquad  \text{ if } x \in I_n =  [1-2^{1-n}, 1-2^{-n}),\ n \geq 1.
\end{align}
A \emph{rotated odometer} is obtained as a composition of $\mathfrak{a}$ with a permutation of a partition of $I$ into $q$ intervals of equal length.
Namely, take a positive integer $q$ that is not a power of $2$ (see \cite{BL23} for the case when $q = 2^m$, $m \in \mathbb N$), and a permutation $\pi$ of the symbols $\{0,\dots, q-1\}$.
Divide the unit interval into $q$ equal-sized half-open intervals
$I_j = [j/q, (j+1)/q)$, $j = 0, \dots, q-1$, and define
$$
 R_\pi:I \to I, \qquad  x \mapsto x+\frac{\pi(j)-j}{q} \text{ if } x \in I_j.
$$
For instance, a rotation by $p/q$, for $0 < p \leq q-1$, is an example of such a map, but not every $R_\pi$ is induced by a rotation.

\begin{definition}\label{defn-rotated}
A \emph{rotated odometer} is an IIET given by the composition 
  $$F_\pi = \mathfrak{a} \circ R_\pi : I \to I,$$ where $\pi$ is a permutation of $q$ symbols.
\end{definition}

The dynamics of rotated odometers was studied in \cite{BL23-JLMS} using methods of symbolic dynamics, see Section~\ref{sec:rotated} for details. In particular, associated to the aperiodic subsystem of a rotated odometer, there is an eventually periodic sequence of substitutions $(\chi_k)_{k \geq 1}$, and $q \times q$-substitution matrices $M_k$, $k \geq 1$, which have non-negative integer entries. In this paper, we restrict to the case when the sequences of substitutions and the associated matrices are stationary, i.e., $\chi_k = \chi$ and $M_k = M$ for all $k \geq 1$, for some substitution $\chi$ and a $q \times q$ matrix $M$.
Since all rotated odometers have an eventually periodic sequence of substitutions, they all have a return map described by a stationary sequence.
We use the properties of $\chi$ and $M$ in our analysis of the $\mathbb{Z}$-extensions of rotated odometers below.

\begin{remark}\label{remark-translation}
As shown in \cite[Theorem 1.1]{BL23-JLMS}, every rotated odometer in Definition~\ref{defn-rotated} can be realized as the first return map of a flow of rational slope on an infinite type translation surface $L$, where $L$ has finite area, infinite genus, and at least one and at most a finite number of punctures,
see Section~\ref{sec-lochness}; see \cite{Rbook} for more on translation surfaces of infinite type.
\end{remark}

\subsection{Non-ergodicity of $\mathbb{Z}$-extensions of rotated odometers}
Throughout this article, we use a probability measure $\mu$ on (a subset of) $I = [0,1]$ (most frequently, $\mu$ is Lebesgue) that is preserved by $F_\pi$.
We denote the corresponding measure-preserving dynamical system  by $(I,F_\pi,\mu)$. We use the notation $(I,F_\pi)$ when we only speak about the topological properties of the system.

To define the skew-product of a rotated odometer $(I,F_\pi,\mbox{Leb})$, let $\psi: I \to \mathbb{Z}$ be a {\em skew-function} given by
\begin{align}\label{skew-function}
  \psi(x) = \begin{cases}
              \ 1,  & 0 \leq x < \frac{1}{2}, \\
              -1, &  \frac{1}{2} \leq x < 1.
            \end{cases}
\end{align}
The skew-product of $F_\pi$ and $\psi$ is given by
\begin{equation}\label{eq-skewprod}
T_\pi: I \times \mathbb{Z} \to I \times \mathbb{Z}, \qquad (x,n) \mapsto (F_\pi(x), n+\psi(x)),
\end{equation}
with invariant measure $\mbox{Leb} \otimes \nu$, where $\nu$ is the counting measure on $\mathbb{Z}$. The skew-product \eqref{eq-skewprod} is also called the $\mathbb Z$-\emph{extension} of $F_\pi$ with skew-function $\psi$.

\begin{remark}
The choice of the skew-function $\psi$ in \eqref{skew-function} is motivated by the interpretation of rotated odometers as the first return maps of flows of rational slope on a translation surface, see Remark~\ref{remark-translation}. From this point of view, the skew-product \eqref{eq-skewprod} is interpreted as the first return map of the lifted flow on a $\mathbb{Z}$-fold covering space of the translation surface in Remark~\ref{remark-translation}, to the section $I \times \mathbb{Z}$, see Section \ref{sec-staircase} for details.
The choice of the points of discontinuity for the step function $\psi$ is motivated by the geometry of the staircase.
\end{remark}

\begin{problem}
Is the skew-product dynamical system \eqref{eq-skewprod} recurrent or transient?
\end{problem}

A necessary (but not sufficient) condition for the skew-product to be ergodic is that it is recurrent to a subsection $I \times \{0\}$, i.e., $\mbox{Leb}$-a.e.\ $x \in I \times \{0\}$ returns to $I \times \{0\}$
infinitely often.
Since the invariant measure $\mbox{Leb} \otimes \nu$ is infinite, this recurrence doesn't follow directly
from Poincar\'e recurrence, but by a result of Atkinson \cite{At1976}, $T_\pi$ is recurrent
if and only if $\int_I \psi \, d  \mbox{Leb} =0$.

If $\pi$ is the identity permutation, then $F_\pi  = \mathfrak{a}$, and we denote the skew-product map by $T_\mathfrak{a}$. In this case $(I \times \mathbb{Z}, T_\mathfrak{a})$ is the first return map to $I \times \mathbb{Z}$ of the vertical flow on the staircase, see Section \ref{sec-staircase}.
This flow itself, i.e., the $\mathbb{R}$-extension of the von Neumann-Kakutani map $\mathfrak{a}$, has been studied for various classes of skew-functions, see for instance \cite{Hellekalek1987,Pask1991} and references therein. In particular, \cite{Hellekalek1987} studies $\mathbb{R}$-extensions when the skew-product function is a step function with two discontinuities. For $\psi$ as in \eqref{skew-function}, \cite[Theorem 2]{Hellekalek1987}  implies that the first return map of the vertical flow to the section $I \times \mathbb Z$ is not ergodic. 

\begin{problem}\label{prob-non-ergodic}
Show that, for any $q \ne 2^m$, $m \in \mathbb{N}$, and any permutation $\pi$ of $q$ intervals, the skew-product $(I \times \mathbb{Z}, T_\pi, \mbox{Leb} \otimes \nu)$ is not ergodic.
\end{problem}

Since the family of all rotated odometers is rather large, and they exhibit diverse properties, we cannot give an affirmative answer to Problem~\ref{prob-non-ergodic} for a $\mathbb{Z}$-extension of any rotated odometer. Rather, we give a few criteria which allow us to rule out ergodicity, and present examples which satisfy these criteria. So far we have not found an example of an ergodic $\mathbb{Z}$-extension in this context.

\begin{remark}
Ergodicity of parallel flows on $\mathbb Z^d$-extensions of translation surfaces and, more generally, of $\mathbb Z^d$-extensions of interval exchange transformations have been an active area of research recently, see, for instance, \cite{FU14,FH18} and references therein. In particular, a criterion for non-ergodicity for some infinite billiards and $\mathbb Z$-covers of translation surfaces was developed in \cite{FU14}. A sufficient criterion for non-ergodicity of flows on $\mathbb Z^d$-covers of compact translation surfaces was developed in \cite{FH18}. Our criteria apply to the class of $\mathbb Z$-extensions of translations surfaces not covered by the assumptions in \cite{FU14,FH18}, namely those which have wild singularities, and whose first return map to $I \times \mathbb Z$ is a $\mathbb Z$-extension of an infinite IET.
\end{remark}

Our approach to this problem is as follows.
In Section~\ref{sec:renorm} we recall from \cite{BL23-JLMS} that
there is a sequence of substitutions $(\chi_k)_{k \geq 1}$, associated to a rotated odometer $(I,F_\pi, \mbox{Leb})$, which assigns to each $i \in \mathcal{A} = \{0,\ldots, q-1\}$ a finite word $\chi_k(i)$ in the alphabet $\mathcal{A}$.
Since there are only finitely many permutations on $q$ letters,
$(\chi_k)_{k \geq 1}$ is (pre-)periodic.
If this sequence is constant, then the rotated odometer $(I,F_\pi)$ is called \emph{stationary}, and we write $\chi_k = \chi_1 = \chi$. 

\begin{assumption}\label{assum-stationarycovering}
Let $(I,F_\pi)$ be a rotated odometer. In what follows, we assume that:
\begin{enumerate}
\item The rotated odometer $(I,F_\pi)$ is stationary,
\item Lebesgue measure $\mbox{Leb}$ is ergodic for $(I,F_\pi)$. Note that we do not assume that $\mbox{Leb}$ is the unique invariant measure, and in fact most our examples are not uniquely ergodic.
\end{enumerate}
\end{assumption}

Using the coding procedure described in Proposition~\ref{lem:periodicpermutations-1}, for each $x \in I$ we can describe a finite piece of its orbit of arbitrary length as a subword of a substitution word $\chi^r(a)$, for some $a \in \mathcal{A}$, and $r \in \mathbb{N}$ large enough. Using the skew-function $\psi$ from \eqref{skew-function}, we can associate to $\chi(a)$ its \emph{weight} $\psi(\chi(a))$ which is an integer, see Section~\ref{sec:doub} for details. Similarly we can associate weights $\psi(\ell_j)$ to left eigenvectors $\ell_j$, $0 \leq j \leq q-1$, of the matrix $M$ describing the substitution. These weights play a central role in our proofs.

\begin{remark}
The standard approach to study the dynamics of systems given by sequences of substitutions is by associating to them an $S$-\emph{adic subshift}, see Section~\ref{subsec:subst} for details. In an $S$-adic shift, the phase space is obtained as the closure of a sequence
$\rho = \lim_k \chi_1 \circ \dots \circ \chi_k(a_k)$ in $\mathcal A^{\mathbb N_0}$ with cylinder set topology, and the dynamics is given by a shift action. For rotated odometers, $\rho$ corresponds to the orbit of $0$ which is transitive in the minimal subset of $I$, but not transitive in $I$, unless $(I,F_\pi,\mbox{Leb})$ is uniquely ergodic, which is not always the case.  That is the reason we do not use the formalism of $S$-adic shifts in our results.
\end{remark}

We now state two criteria for non-ergodicity of rotated odometers. The first one is formulated in terms of the weights of the substitution words.

\begin{theorem}\label{thm-1.0-non-essential}
Let $(I,F_\pi,\mbox{Leb})$ be a stationary rotated odometer. Let $T_\pi: I \times \mathbb{Z} \to I \times \mathbb{Z}$ be its $\mathbb{Z}$-extension with skew-function $\psi$ as in \eqref{skew-function}, and  $\boldsymbol{d} := \gcd\{ \psi(\chi(a)) : a \in \mathcal{A}\}$. If $\boldsymbol{d} \geq 2$, then the skew-product is not ergodic.
\end{theorem}

Example~\ref{ex-q11} of a rotated odometer with $q = 11$ satisfies the hypothesis of Theorem~\ref{thm-1.0-non-essential} with ${\bf d} = 2$, and so it is not ergodic. Example~\ref{ex-9-2ue} with $q=9$ also satisfies Theorem~\ref{thm-1.0-non-essential} with ${\bf d} = 2$. For the latter, we can give a more precise description of the ergodic properties using Theorem~\ref{thm-uniqergod-pisot} below.

\begin{theorem}\label{thm-uniqergod-pisot}
 Let $(I,F_\pi,\mbox{Leb})$ be a stationary rotated odometer.  Let $T_\pi: I \times \mathbb{Z} \to I \times \mathbb{Z}$ be its $\mathbb{Z}$-extension with skew-function $\psi$ as in \eqref{skew-function}. Assume that for all eigenvalues
 $\lambda_j$ of the associated matrix $M$ of $\chi$ with norm $|\lambda_j| \geq 1$, the left eigenvectors have weights
 $\psi(\ell_j) = 0$, with the exception of one, say, for $\lambda_c > 1$, which is Pisot. Suppose the algebraic and the geometric multiplicities of $\lambda_c$ are equal. 
Then Lebesgue measure $\mbox{Leb}$ of the $\mathbb{Z}$-extension $T_\pi$ has infinitely many ergodic components.
\end{theorem}

Theorem~\ref{thm-uniqergod-pisot} applies to uniquely ergodic rotated odometers in Examples~\ref{ex-3-1ue} and \ref{ex-3-2ue} for $q=3$, Example~\ref{ex-5-3nue} for $q=5$. In these examples, there are two eigenvalues of algebraic multiplicity $1$ outside of the unit circle, and for the largest one $\psi(\ell_0) = 0$ (this is always true for rotated odometer with ergodic Lebesgue measure). Theorem~\ref{thm-uniqergod-pisot} also applies to Example~\ref{ex-9-2ue} for $q=9$, where $(I,F_\pi,\mbox{Leb})$ is uniquely ergodic, and the second largest eigenvalue $\lambda_1$ is Pisot and has geometric and algebraic multiplicity $2$. Example~\ref{ex-5-3nue} for $q = 5$ is not uniquely ergodic, and satisfies the assumptions of the theorem. We refer the reader to Section~\ref{sec:exam} for further examples, to which Theorems~\ref{thm-1.0-non-essential}  and  \ref{thm-uniqergod-pisot} are not applicable.

\subsection{Diffusion coefficient}

The spread of particles over time, for example in a Brownian motion, is called diffusion, and if the (average or maximal) displacement of particles happens according to some power law, the exponent is called the
{\em diffusion coefficient}.
For Brownian motion, the average displacement is the square-root of time, so the diffusion coefficient
is $\gamma = \frac12$, and this is ``the standard''.
Non-standard diffusion is referred to as {\em anomalous diffusion}, and more precisely, it is called {\em super-diffusion} if the diffusion coefficient $\gamma > \frac12$
and it is called {\em sub-diffusion} if $\gamma < \frac12$.
It is known that the Lorentz gas with circular obstacles and finite horizon (i.e., a particle will always hit an obstacle in bounded time) is diffusive \cite{BS80}, while the Lorentz gas with infinite horizon is super-diffusive \cite{SV07}, but only by a thin margin, because the diffusion coefficient is $\frac12$ also here, but there is a logarithmic correction factor
that makes the diffusion faster than for Brownian motion.
Ehrenfest's wind-tree model is similar to the Lorentz gas model, except that the obstacles are now aligned rectangles (or more general polygonal shapes, see \cite{DZ2020}) situated at the lattice points of $\mathbb{Z}^2$. The diffusion depends now
on the direction of the flow. In the model with square obstacles, the diffusion coefficient for the typical direction is $\gamma = \frac{2}{3}$ \cite{DHL}, and this is found as the quotient of the largest two eigenvalues of a particular matrix. For more complicated polygonal shapes, the diffusion coefficients in the typical direction can be very different. For instance, \cite{DZ2020} considers the a family of models with a certain geometric shape of an obstacle, where the diffusion coefficient depends only on the number of the corners of the obstacle, and tends to zero as the number of corners increases. On the other hand, \cite{DZ2020} also describes a model where the diffusion coefficient is very close to $1$.

For $\mathbb{Z}$-extensions, the $n$-steps displacement is $\xi(T_\pi^n(x,k)) - \xi(x,k)$ for the projection $\xi(x,k) = k$.
Clearly, the displacement is independent of $k \in \mathbb{Z}$,
and thus the diffusion coefficient, as function of $x$, can be defined as
\begin{align}\label{eq-diff-defn}
\gamma(x) = \limsup_{n \to \infty}
\frac{ \log (\xi(T_\pi^n(x,k)) - k )}{\log n}.
\end{align}
Unless $F_\pi$ is uniquely ergodic, the dependence of $\gamma(x)$
on the point $x$ cannot be avoided, so at best we can speak of the a.s.\ diffusion coefficient
with respect to an ergodic measure (which will be Lebesgue, if $\chi$ is covering, see Definition~\ref{def-covering} below).
Our main result about $\gamma(x)$ is the following.

\begin{theorem}\label{thm-diffusion}
Let $(I,F_\pi,\mbox{Leb})$ be a stationary rotated odometer. Suppose the associated matrix is diagonalizable, the largest eigenvalue $\lambda_0$ has multiplicity $1$, and the second largest eigenvalue satisfies $|\lambda_1| >1$, then for the corresponding $\mathbb{Z}$-extension,
the diffusion coefficients satisfy 
$$
\gamma(x) \leq \gamma_0 := \frac{\log |\lambda_1|}{\log \lambda_0}
\text{ for Leb-a.e. } x.
$$
\end{theorem}

\subsection{Discrepancy} A sequence of points $(x_i)_{i \geq 1}$ in $I$ is called uniformly distributed if, in the limit, the average of times the sequence hits a subinterval $J$ of $I$ is equal to the length of $J$. It is well-known that, for a measure-preserving system with an ergodic invariant probability measure, the orbit of almost every point is uniformly distributed. Discrepancy is a quantitative characteristic of a sequence $(x_i)_{i \geq 1}$, which can be understood as characterizing the properties of the uniform distribution, i.e., how far it is from the ideal one. A classical reference for the results on discrepancy is the book \cite{KN}. In this paper, rigorous definitions of the concepts mentioned above are given in Section~\ref{sec:discr}.

For substitutions, the discrepancy of sequences arising as fixed points of primitive substitutions and the corresponding shifts were studied by Adamczewski \cite{Ad04}. His results \cite{Ad04} were extended to $S$-adic transformations in \cite{BD14}. Since our rotated odometers are often described by substitution matrices which are not primitive, we include the following estimate for the discrepancy of orbits of rotated odometers.

\begin{theorem}\label{thm:discr}
Suppose that $F_\pi$ is covering and stationary,
 the associated matrix is diagonalizable, and the largest eigenvalue $\lambda_0$ has multiplicity $1$.
Then for Lebesgue-a.e.\ $x$, there is $C = C_x$ such that
the $F_\pi$-orbit of $x$ has discrepancy
$\mathfrak{D}_R \leq C_x \cdot R^{\gamma_0-1}$
for $\gamma_0 := \max\left\{ \frac{\log |\lambda_1|}{\log \lambda_0}, 0 \right\}$.
\end{theorem}

 The rest of the paper is organized as follows. Section \ref{sec-surfaces} briefly explains the relation between rotated odometer and their skew-product extensions, and the first return maps of flows on translation surfaces.
 Section~\ref{sec:rotated} recalls preliminaries, mostly from \cite{BL23-JLMS} but also about essential values from \cite{S77}.
 Section~\ref{sec:discr} computes the discrepancy of the rotated odometers, as well as the diffusion coefficients of the skew-products.
 Recurrence and (non-)ergodicity of the skew-products are covered in Section~\ref{sec:non-erg}.
 In Section~\ref{sec:exam}, we list the examples that illustrate our theorems.
 
  \section{Skew-products of rotated odometers as first return maps}\label{sec-surfaces}
 
 In this section we briefly describe how rotated odometers and their skew-products can be realized as first return maps of flows on translation surfaces.
 
 \subsection{Flows on a finite area translation surface}\label{sec-lochness}
 To construct a translation surface of finite area and infinite genus, we identify a pair of the opposite edges of the unit square (minus a countable collection of points) using the von Neumann-Kakutani map, and another pair by an IET of a finite number of intervals. This construction is similar to the one of the Chamanara, or the baker's surface, see for instance \cite{Rbook}, except in the baker's surface the second pair of edges is also identified via the von Neumann-Kakutani map. Due to this second identification, the baker's surface has more symmetries, which allows, for instance, to compute Veech group for this surface \cite{DHV-book}. The surface $L$ constructed below lacks such symmetries.
 
 More precisely, let $\tau: I \to I$ be an exchange of a finite number of intervals of equal length. Denote by $D_a$ and $D_\tau$ the sets of discontinuities of $\mathfrak{a}$ and $\tau$ respectively, and let 
  \begin{align}S_{\mathfrak{a}} = \{(x,1), (\mathfrak{a}(x),0) \mid x \in D_{\mathfrak{a}} \cup \{0,1\}\}, \\
  S_\tau = \{(0, y), (1, \tau(y)) \mid y \in D_\tau \cup \{0,1\}\}.
  \end{align}

Define an equivalence relation on $[0,1] \times [0,1]$ by
  \begin{align}\label{eq-simrel}(x,y) \sim (x',y') \textrm{ if and only if } y=1, y'=0, x'= \mathfrak{a}(x) \textrm{ or }x=0, x'= 1, y'= \tau(y),\end{align}
see Figure \ref{fig:surfaces}, left. Let 
  \begin{align}\label{eq-nocorners}L'  =  \left( [0,1] \times [0,1] \right)\ \setminus \ \left( S_a \cup S_\tau\right), \end{align}
  then the quotient surface $L = L'/\sim$ is a smooth non-compact surface.

\begin{figure}[ht]
\begin{center}
\begin{minipage}{0.4\textwidth}
\begin{tikzpicture}[scale=0.6]
\draw[-, draw=black] (0,8)--(7.6,8)node[pos=0.25,anchor=north]{$I_1$} 
node[pos=0.66,anchor=north]{$I_2$} node[pos=0.85,anchor=north]{$I_3$} 
node[pos=0.95,anchor=north]{$I_4$} node[pos=0,anchor=east]{$1$}; 
\draw[dotted, draw=black] (7.6,8)--(8,8); 
\draw[-, draw=black] (0,8)--(0,0)node[pos=0.5,anchor=east]{$J_1$}; 
\draw[-, draw=black] (8,0)--(8,8)node[pos=0.5,anchor=west]{$\tau(J_1)$}; 
\draw[-, draw=black] (0.4,0)--(8,0) node[pos=0.75,anchor=south]{$\mathfrak{a}(I_1)$} 
node[pos=0.36,anchor=north]{$\mathfrak{a}(I_2)$} node[pos=0.14,anchor=north]{$\mathfrak{a}(I_3)$} 
node[pos=0.01,anchor=south]{$\mathfrak{a}(I_4)$}node[pos=-0.04,anchor=east]{$0$} node[pos=1,anchor=west]{$1$}; 
\draw[dotted, draw=black] (0,0)--(0.4,0); 
\draw[-,draw=black] (4,-0.1)--(4,0.1);
\draw[-,draw=black] (2,-0.1)--(2,0.1);
\draw[-,draw=black] (1,-0.1)--(1,0.1);
\draw[-,draw=black] (0.5,-0.1)--(0.5,0.1);
\draw[-,draw=black] (4,7.9)--(4,8.1);
\draw[-,draw=black] (6,7.9)--(6,8.1);
\draw[-,draw=black] (7,7.9)--(7,8.1);
\draw[-,draw=black] (7.5,7.9)--(7.5,8.1);
\draw[-,thick,dashed,draw=blue] (2.1,0)--(4.8,8);
\draw[-,thick,dashed,draw=blue] (2.1,0.1)--(4.8,0.1) node[pos=0.5,anchor=south]{\tiny $+\frac{p}{q}$};
\draw[-,thick,dashed,draw=blue] (4.8,0)--(4.8,8);
\end{tikzpicture}
\label{fig:lochness}
\end{minipage}
\begin{minipage}{0.1\textwidth}
\end{minipage}
\qquad \qquad \quad
\begin{minipage}{0.3\textwidth}
\begin{tikzpicture}[scale=0.19]
\draw[-,draw=green] (12,24)--(12,32)--(20,32)--(20,24)--(12,24);
\draw[-,draw=green] (11.5,26.5)--(12.5,28.5); \draw[-,draw=green] (19.5,26.5)--(20.5,28.5);
\draw[-,draw=green] (16.5,31.5)--(17.5,32.5); \draw[-,draw=green] (18.5,15.5)--(19.5,16.5);
\draw[-,draw=green] (12,32) circle (0.15);
\draw[-,draw=green] (16,32) circle (0.15);
\draw[-,draw=green] (18,32) circle (0.15);
\draw[-,draw=green] (19,32) circle (0.15);
\draw[-,draw=green] (19.5,32) circle (0.15);
\draw[-,draw=green] (20,32) circle (0.15);
  \draw[-,draw=green] (20,16) circle (0.15);
\draw[-,draw=green] (16,16) circle (0.15);
\draw[-,draw=green] (18,16) circle (0.15);
\draw[-,draw=green] (17,16) circle (0.15);
\draw[-,draw=green] (16.5,16) circle (0.15);
\draw[-,draw=green] (19,16) circle (0.15);
\draw[-,draw=blue] (16,16)--(16,24)--(24,24)--(24,16)--(16,16);
\draw[-,draw=blue] (15.5,18.5)--(16.5,20.5); \draw[-,draw=blue] (23.5,18.5)--(24.5,20.5);
\draw[-,draw=blue] (20.5,23.5)--(21.5,24.5); \draw[-,draw=blue] (22.5,7.5)--(23.5,8.5);
\draw[-,draw=blue] (16,24) circle (0.15);
\draw[-,draw=blue] (20,24) circle (0.15);
\draw[-,draw=blue] (22,24) circle (0.15);
\draw[-,draw=blue] (23,24) circle (0.15);
\draw[-,draw=blue] (23.5,24) circle (0.15);
\draw[-,draw=blue] (24,24) circle (0.15);
  \draw[-,draw=blue] (24,8) circle (0.15);
\draw[-,draw=blue] (20,8) circle (0.15);
\draw[-,draw=blue] (22,8) circle (0.15);
\draw[-,draw=blue] (21,8) circle (0.15);
\draw[-,draw=blue] (20.5,8) circle (0.15);
\draw[-,draw=blue] (23,8) circle (0.15);
\draw[-,draw=black] (20,8)--(20,16)--(28,16)--(28,8)--(20,8);
\end{tikzpicture}
\end{minipage}
\caption{\footnotesize Left: The unit square with identifications given by \eqref{eq-simrel} for $\tau = {\rm id}$. Dashed lines represent a straight-line flows line at an angle $\theta = \tan^{-1}\frac{q}{p}$ and its decomposition as the horizontal translation by $\frac{p}{q}$ and the vertical translation by $1$. 
Right: Three steps of the staircase with identifications.}
\label{fig:surfaces}
\end{center}
\end{figure} 

Topologically, surfaces are classified by \emph{genus} and the \emph{number of ends}. Intuitively, the genus is the number of handles from which the surface can be assembled, and the number of ends is the number of ways to go to infinity in the surface, see \cite{Rbook,DHV-book} for precise definitions. Intuitively, an end is planar if it is possible to cut off the part of the surface containing the end in such a way that the cut-off portion has genus zero, and an end is non-planar otherwise. In translation surfaces with compact metric completion, planar ends correspond to cone-angle singularities, while non-planar ends correspond to wild singularities, see \cite{Rbook}. 

\begin{theorem}\cite[Theorem 1.1]{BL23-JLMS}\label{thm-lochness}
The surface $L$ constructed above has finite area, infinite genus, precisely one non-planar end, and at most a finite number of planar ends.
\end{theorem}
The number of planar ends depends on the finite IET $\tau$, see \cite{BL23-JLMS} for details.

The rotated odometers defined in the introduction can be represented as first return maps of parallel flows of rational slope on the surface $L$ as follows. Let $P$ be the image of $(0,1)$ in $L$ after taking the quotient by $\sim$ and removing the points of discontinuity. Then $P$ is a Poincar\'e section for a parallel flow of rational slope $\frac{p}{q}$, $p,q \in \mathbb{N}$, with the first return map given by the composition $F = \mathfrak{a} \circ R: P \to P$, where $R$ is a finite interval exchange transformation of $q$ intervals of equal length, depending on the slope of the flow-lines and the transformation $\tau$. The latter, with the measure on $P$ induced from Lebesgue measure, is measure-theoretically isomorphic to $(I,F,\mbox{Leb})$.

\subsection{The staircase} \label{sec-staircase}

The skew-product systems over rotated odometers can be realized as first return maps of parallel flows of a $\mathbb{Z}$-cover of the finite area surface $L$ in Section \ref{sec-lochness}, called a \emph{staircase}. 

Take the product $\mathcal{L}' = L' \times \mathbb{Z}$, where $L'$ is defined by \eqref{eq-nocorners}. Define an equivalence relation $\sim_{st}$ on $\mathcal{L}'$ by
$$(x,y,n) \sim_{st} (x',y',n') \textrm{ if and only if }\left \{ \begin{array}{ll} \textrm{either} & y =1, y'= 0, x'= \mathfrak{a}(x), n'= n+ \psi(x), \\ \textrm{or} & x=0,x'=1,y'= \tau(y), n= n'. \end{array} \right.$$
Then set $\mathcal{L} = \mathcal{L}'/ \sim_{st}.$  Geometrically we can represent $\mathcal{L}$ as an infinite collection of squares in Figure \ref{fig:surfaces}, right (there $\tau = {\rm  id}$), stacked on top of each other with a shift by $\frac{1}{2}$, and such that the right half of the upper edge of $(L',n)$ is identified with the left half of the lower edge of $(L',n-1)$ using the von Neumann-Kakutani map, see Figure~\ref{fig:surfaces}, right. Here values of $\mathbb{Z}$ are increasing in the upward direction. 

Clearly there is an infinite-to-one projection
 \begin{align} {\rm pr}_1: \mathcal{L}' = L' \times \mathbb{Z} \to L', \quad (x, y, n) \mapsto (x,y), \end{align}
and the equivalence relation $\sim_{st}$ descends under  the projection to the equivalence relation $\sim$ in \eqref{eq-simrel}. Thus there is an induced infinite-to-one projection $ {\rm pr}_1: \mathcal{L} \to L $. Using the standard methods as in the proof of \cite[Theorem 1.1]{BL23-JLMS}, or in \cite{Rbook,DHV-book}, one can show that the staircase $\mathcal L$ is a translation surface of infinite area of infinite genus. For a parallel flow $f^t:L \to L$ of rational slope with first return map $F_\pi$ to the Poincar\'e section $P$ as in Section \ref{sec-lochness}, the lift of $P$ along the map ${\rm pr}_1$ is the set $\mathcal{P} = P \times \mathbb{Z}/ \sim_{st}$, and the lift $\widehat f^t$ of $f^t$ is the straight-line parallel flow on $\mathcal L$ of the same slope as $f^t$. The first return map of $\widehat f^t$ to $\mathcal{P}$ is measure-theoretically isomorphic to the skew-product $T_\pi: I \times \mathbb Z \to I \times \mathbb Z$\eqref{eq-skewprod}, with respect to the measure $\mbox{Leb} \times \nu$, where $\nu$ is the counting measure on $\mathbb Z$.

\begin{remark}
Staircases (in other words, $\mathbb Z$-covers) over compact translation surfaces with finite number of punctures were studied in \cite{HW}, which considered the topology of such staircases, the recurrence of flows in terms of their homology, and gave a sufficient condition on the Veech group of the staircase to have certain properties. The family of staircases considered in our paper falls within the family considered in \cite{HW}. However, the Veech group of the surface $L$ in Section \ref{sec-lochness} is not known, and consequently we cannot say anything about the Veech group of the staircase $\mathcal{L}$ described in this section.
\end{remark}

 \section{Preliminaries of rotated odometers and essential values}\label{sec:rotated}
 We recall some results of \cite{BL23, BL23-JLMS} on the properties of rotated odometers, and techniques to study them.
 Namely, \cite[Theorem 1.2]{BL23-JLMS} proves that a rotated odometer $(I,F_\pi)$ as in Definition~\ref{defn-rotated} has the following properties:
\begin{enumerate}
\item There exists a decomposition $I = I_{per} \cup I_{np}$, with $I_{per}$ possibly empty, such that every point in $I_{per}$ is periodic under $F_\pi$, $0 \in I_{np}$, and the restriction $F_\pi: I_{np} \to I_{np}$ is invertible at every point except $0$, and has no periodic points.
\item There is a unique minimal subsystem $(I_{min},F_\pi)$ of $(I_{np},F_\pi)$.
\item If $I_{per}$ is non-empty, then it is a finite or a countable union of half-open subintervals of $I$.
\item The minimal subsystem $(I_{min},F_\pi)$ is uniquely ergodic.
\item The aperiodic subsystem $(I_{np},F_\pi)$ is ergodic, with at most $q$ invariant ergodic probability measures.
\end{enumerate}

\subsection{Substitutions}\label{subsec:subst}
To study the dynamics of $(I,F_\pi)$ further, we code the orbits of points in $(I,F_\pi)$ by words in the alphabet $\mathcal{A} = \{0,\ldots,q-1\}$.
Let $\mathcal{A}^*$ be the set of all non-empty words of finite length
in this alphabet, and let $\Sigma = \mathcal{A}^{\mathbb{N}_0}$ be the set of one-sided infinite sequences in this alphabet.

\begin{definition}\label{def:substitution}
A {\em substitution} $\chi: \mathcal{A} \to \mathcal{A}^*$ is a map that assigns to every $a \in \mathcal{A}$ a single word
$\chi(a) \in \mathcal{A}^*$. It
extends to $\mathcal{A}^*$ and $\Sigma$ by concatenation:
$$
\chi(a_1a_2\dots a_r) = \chi(a_1)\chi(a_2)\dots \chi(a_r), \, r \geq 1.
$$
The $q\times q$ matrix $M$, where the $(i,j)$-th entry is the number of letters $j$
in $\chi(i)$, is called the \emph{associated matrix} of $\chi$.
\end{definition}

\begin{definition}\label{def:primproper}
A substitution $\chi: \mathcal{A} \to \mathcal{A}^*$ is:
\begin{itemize}
\item {\em primitive}, if there is $r \geq 1$ such that for all $i\in \mathcal{A}$,
$\chi^r(i)$ contains every letter in $\mathcal{A}$.
\item {\em proper} if there exist two letters $a,b \in \{ 0, \dots, q-1\}$
such that for all $i \in \{ 0,\dots, q-1\}$, the first letter of $\chi(i)$ is $a$
and the last letter of $\chi(i)$ is $b$. 
\end{itemize}
\end{definition}

The substitutions we consider are  always \emph{proper} and sometimes also \emph{primitive}.
Recall that a square matrix $M$ is \emph{primitive} if it has a power with strictly positive entries. Thus a substitution $\chi$ is primitive as in Definition~\ref{def:primproper} if and only if its associated matrix $M$ is primitive.

If $\chi(a)$ starts with $a$, we get a fixed point of $\chi$ which (unless $\chi(a) = a$) is an infinite sequence
\begin{align}\label{eq-rho}
\rho = \rho_0\rho_1\rho_2\rho_3 \dots = \lim_{j \to \infty} \chi^j(a) \in \Sigma.
\end{align}
For sequences $s = (s_k) \in \Sigma$, define the \emph{left shift} by
\begin{equation}\label{eq-shiftmap}
\sigma: \Sigma \to \Sigma, \qquad  s_0s_1 s_2 \ldots \mapsto s_1s_2 s_3 \ldots.
\end{equation} 
The dynamical system $(X_\rho,\sigma)$
where $X_\rho = \overline{\{\sigma^n(\rho) : n \geq 0\}}$ is called a \emph{substitution subshift}.

\subsection{Renormalization and symbolic dynamics}\label{sec:renorm}

We now briefly describe the renormalization procedure developed in \cite{BL23-JLMS}, which allows us to associate to the rotated odometer $(I,F_\pi)$ a substitution dynamical system. For that we restrict to the aperiodic subsystem $(I_{np},F_\pi)$ of the rotated odometer $(I,F_\pi)$.

Let $q \geq 2$ and let $F_\pi = \mathfrak{a} \circ R_\pi$ be a rotated odometer as in Definition~\ref{defn-rotated}. Let 
 \begin{equation}\label{eq-defnN}
N = \min\{ n \geq 1 : 2^n \geq q\}
\quad \text{ and } \quad 
L_k = [0,2^{-kN}). 
\end{equation}
Denote by $\mathcal{P}_{0,q}$ the partition of $I = [0,1)$ into $q$ equal length intervals, which are the intervals of continuity for $R_\pi$ in the definition of $F_\pi$. The half-open interval $L_k$ is called the \emph{$k$-th section}.

For each $q \geq 1$ and each $k \geq 1$, we define a partition $\mathcal{P}_{kN,q}$ of the unit interval into $q2^N$ intervals of equal length $q^{-1}2^{-N}$. This partition refines $\mathcal P_{mN,q}$, $0 \leq m \leq k-1$, and induces the partition $\mathcal{P}_{kN,q}^{cod}$ of the $k$-section $L_k$ into $q$ subintervals $L_{k,i}$, $0 \leq i \leq q-1$. We recall the following result.

\begin{prop}\cite[Proposition 5.11, Lemma 5.15]{BL23-JLMS}\label{lem:periodicpermutations-1}
Let $F_{\pi,k}: L_k \to L_k$, $k \geq 1$ be the first return maps of the rotated odometer $F_\pi: I \to I$.

Then for $k \geq 1$ there exist permutations $\pi_k$ of $q$ symbols, and finite IET $R_{\pi,k}:L_k \to L_k$ of the partitions $\mathcal{P}_{kN,q}^{cod}$, defined by $\pi_k$, such that 

\begin{enumerate}
\item \label{it1-nn} $F_{\pi,k} = \mathfrak{a}_k \circ R_{\pi,k}$, where $\mathfrak{a}_k$ is a scaled copy of the von Neumann-Kakutani map $\mathfrak{a}$, given by
$$
\mathfrak{a}_k(x) = \frac{1}{2^{kN}}\, \mathfrak{a} \left(2^{kN} x \right).
$$
\item \label{it2-nn} The sequence $(\pi_k)_{k \geq 1}$ is pre-periodic.
\item If $(\pi_k)_{k \geq 1}$ is strictly pre-periodic, and $S = \{\pi_1,\ldots, \pi_{k_0-1}\}$ is the pre-periodic part of the sequence, then none of the permutations in $S$ occurs in the periodic part.
\item For each set $L_{k,i}$, $0 \leq i \leq q-1$ of the coding partition $\mathcal P_{kN,q}$, the order in which $L_{k,i}$ visits the sets of $\mathcal P_{(k-1)N,q}$ under the first return map $F_{\pi,k}: L_{k-1} \to L_{k-1}$ defines a substitution word $\chi_k(i)$ with the alphabet $\mathcal{A} = \{0,\ldots, q-1\}$. Thus associated to each section $L_k$, $k \geq 1$, there is a substitution $\chi_k: \mathcal{A} \to \mathcal{A}^*$.
\item  Every substitution in the sequence $(\chi_k)_{k \geq 1}$ is proper, namely, every word $\chi_k(i)$ starts with $0$, and ends 
 with $b_k$, depending only on $k$ and not on $i \in \{0,\ldots,q-1\}$.
\end{enumerate}
\end{prop}

Thus for each $i \in \mathcal{P}_{kN,q}$ the word $\chi_k(i)$ is the itinerary of the half-open interval $L_{k,i}$ in $L_{k-1}$, with respect to the partition $\mathcal{P}_{(k-1)N,q}$ of $L_{k-1}$, and, inductively, $\chi_1 \circ \chi_2 \circ \cdots \circ \chi_k(i)$ is the itinerary of $L_{k,i}$ in $I = [0,1)$ with respect to the partition $\mathcal{P}_{0,q}$.

We now can proceed similarly to the case of a single substitution in Section~\ref{subsec:subst}.
By Proposition~\ref{lem:periodicpermutations-1}, item (5), the sequence
$$
\rho = \lim_{k \to \infty} \chi_1 \circ \cdots \circ \chi_k(0)
$$
is well-defined and it is an $S$-adic fixed point of $(\chi_k)_{k \geq 1}$. The sequence $\rho$ is also the itinerary of $0$ in $I$ with respect to the partition $\mathcal P_{0,q}$.

We define the  {\em $S$-adic subshift} $(X_\rho,\sigma)$ similarly to Section~\ref{subsec:subst}, formula \eqref{eq-shiftmap}, and the paragraph below \eqref{eq-shiftmap}. Recall  that the rotated odometer $(I,F_\pi)$ has a unique minimal subsystem $(I_{min},F_\pi)$, and the orbit of $0$ under $F_\pi$ lies dense in $I_{min}$. Thus $(X_\rho, \sigma)$ models the dynamics on the minimal set $(I_{min},F_\pi)$. More precisely, there is a homeomorphism onto its image $h: I_{min} \to X_\rho$, such that $X_\rho \setminus h(X_\rho)$ is countable, and $h \circ F_\pi(x) = \sigma \circ h (x)$.

In many examples of rotated odometers $I_{min}$ is a proper subset of the aperiodic subsystem $I_{np}$, so we cannot restrict to considering only $(X_\rho,\sigma)$. Proposition~\ref{lem:periodicpermutations-1} allows us to code the orbits outside the minimal set as well. Before we state the result which we will use, we have to introduce \emph{covering} substitutions.

\begin{definition}\label{def-covering}
Let $F_\pi: I \to I$ be a rotated odometer as in Definition~\ref{defn-rotated}, and let $F_{\pi,k}: L_k \to L_k$ be the sequence of first return maps as in Proposition~\ref{lem:periodicpermutations-1}. Then $F_{\pi,k}$ is {\em covering} if the orbits of the sets in $\mathcal{P}_{kN,q}^{cod}$
visit every element of $\mathcal{P}_{kN,q}$ in $L_{k-1}$, that is,  if
\begin{equation}\label{eq:covering} 
\bigcup_{i=0}^{q-1} \bigcup_{t=0}^{t_{k,i}-1} F^t_{\pi,k-1}(I_{k,i}) = L_{k-1}. 
\end{equation}
Alternatively, $F_{\pi,k}$ (or the associated substitution $\chi_k$) is \emph{covering} 
if $\sum_{i \in \mathcal{A}} |\chi_k(i)| = q2^N$. 

\end{definition}

The importance of {\em covering} substitutions
is underlined by the following result from \cite{BL23-JLMS}:

\begin{prop}\label{prop:LebErg}
 Lebesgue measure is ergodic (and hence $F_\pi$ is aperiodic on $I$) if and only if
 all the substitutions $\chi_k$ (or just $\chi$ if it is stationary) are covering.
\end{prop}

Thus, in our theorems, when we consider a rotated odometer with Lebesgue ergodic measure, we are automatically considering covering substitutions. 

Let $(X_\chi,\sigma)$ denote the smallest subshift such that every subword $W$ that it contains is a subword of
$\chi_{s} \circ \cdots \circ \chi_1(a)$ for some $a \in \mathcal{A}$ and $s \in \mathbb{N}$; it symbolically describes the aperiodic part $(I_{np}, F_\pi)$ of the rotated odometer. 

Lemma~\ref{lemma-code-orbits} is a special case of \cite[Theorem 5.10]{BL23-JLMS}, when every substitution in $(\chi_k)_{k \geq 1}$ is covering.  Since our main results are for ergodic Lebesgue measures, we include a simple proof of the existence of coding in this case for completeness. The proof of the general case in \cite[Theorem 5.10]{BL23-JLMS} requires more involved techniques.

\begin{lemma}\label{lemma-code-orbits}
Let $(I,F_\pi)$ be a rotated odometer, with associated sequence of substitutions $(\chi_k)_{k \geq 1}$. Suppose $\chi_k$ is covering for all $k \geq 1$. Then for any $x \in I$ and $n \geq 1$, there exists $s \in \mathbb N$, $a \in \mathcal{A}$ and a subword $W$ of $\chi_{s} \circ \cdots \circ \chi_1(a)$ such that $W$ is the itinerary of the orbit $\{ F^k_\pi(x) : 0 \leq k \leq n\}$ with respect to the coding partition $\mathcal{P}_{0,q}$. \end{lemma}

\begin{proof}
Let $(L_k)_{k \geq 1}$ be the sequence of sections in Proposition~\ref{lem:periodicpermutations-1}, and let $(\chi_k)_{k \geq 1}$ be the sequence of associated substitutions. Let $x \in I$, and fix an $n \in \mathbb N_0$. For convenience, we introduce the notation
  $$
  \chi^{(k)} (a) = \chi_k \circ \dots \circ \chi_1 (a), \qquad a \in \mathcal A.
  $$
Since every $\chi_k$ is covering for every $k \geq 1$, there is $a_1 \in \mathcal{A}$ such that $x \in F^{j_0}_\pi(L_{k,a_1})$ for some $j_0 \geq 0$. Without loss of generality we may assume that $k$ is so large that $x \notin L_k$, and thus $j_0 >0$.

If $|\chi^{(k)}(a_1)| - j_0 > n$, then the part of the substitution word $\chi^{(k)}(a_1)$, starting from the $j_0$-th symbol and of length $n$, codes the orbit of $x$ of length $n$.

Otherwise we proceed as follows. The orbit of $L_{k,a_1}$ returns to $L_k$ after $|\chi^{(k)}|$ iterations. Then the orbit of $x$ returns to $L_k$ after $j_1  = |\chi^{(k)}(a_1) | - j_0 $ iterations. Since $F_\pi$ is not invertible at $0$, $F^{j_1}(x) \ne 0$, and thus there exists a half-open subinterval $V_1 \subset L_{k,a_1}$ with $x \in F_\pi^{j_0} (V_1)$ and such that $F^{j_0+j_1} (V_1) \subset L_{k,a_2}$ for some $a_2 \in \mathcal A$. If $|\chi^{(k)} (a_1)| - j_0 + |\chi^{(k)}(a_2)| > n$, then the part of the concatenated substitution word $\chi^{(k)}(a_1)\chi^{(k)}(a_2)$, starting from the $j_0$-th symbol and of length $n$, codes the orbit of $x$ of length $n$.

If not, then, continuing by induction, we can find a sequence of half-open subintervals $V_s \subset V_{s-1} \subset \cdots \subset V_1 \subset  L_{k,a_1}$, and a sequence of letters $a_1, a_2, \ldots, a_s$ such that, for any $1 \leq i < s$, $x \in F^{j_0}(V_i)$, and $F^{j_0 + \sum_{m=1}^i j_m} (V_i) \subset L_{k,a_{i+1}}$, and
$$
|\chi^{(k)}(a_1)| - j_0 + \sum_{i=2}^s |\chi^{(k)}(a_i)| > n.
$$
The previous inequality holds for some $s \geq 1$ since all substitution words have non-zero length. Thus the part of the concatenated substitution word $\chi^{(k)}(a_1) \cdots \chi^{(k)}(a_s)$ of length $n$ starting from the $j_0$-th symbol codes the orbit of $x$ of length $n$. The choice of $k$ in this argument is not unique; taking larger values of $k$ will result in a smaller value of $s$. Nevertheless, the coding is unique for a given $x \in I$, since it is always done with respect to the same partition $\mathcal P_{0,q}$. In our argument, we do not need to assume that $x$ is a Lebesgue typical point; rather, any open neighborhood of $x$ contains an interval and so has positive measure; therefore it contains a Lebesgue typical point, whose orbit of length $n$ has the same coding as the orbit of $x$.
\end{proof}

In Proposition~\ref{lem:periodicpermutations-1}, the sequence $(\chi_k)_{k \geq 1}$ is pre-periodic. Therefore,  the properties of the corresponding shift are effectively the same as for the case of a single substitution. Namely,
for $k_0$ the length of the pre-periodic part of $(\chi_k)_{k \geq 1}$ and  $p_0$ the period, set
\begin{equation}\label{eq-bmat}
 M = M_{k_0+p_0} \cdots M_{k_0+1}.
\end{equation}
Thus we may reduce the problem to the study of a single substitution $\chi$ with associated matrix $M$. If $\chi_k = \chi$ for all $k \geq 1$, we say that $\chi_1 = \chi$ is stationary.

\subsection{The associated matrix in Frobenius form}\label{matrix-measures} We now assume that the sequence of substitutions associated to the rotated odometer $(I,F_\pi)$ is stationary.
The matrix $M$ of a covering substitution has column sums equal to $2^N$, so $\lambda_0 = 2^N$ is also the leading eigenvalue, with a positive left eigenvector $\ell_0 = (1,1, \dots, 1)$.
This reflects the constant density of Lebesgue measure.
However, even if Lebesgue measure is ergodic, it needs not be the only ergodic invariant measure.
This can be seen from the {\em Frobenius form} of the matrix $M$:
by conjugating with a permutation matrix $\Pi$, corresponding to relabeling the symbols in $\mathcal{A}$, we can
obtain the form
\begin{equation}\label{eq:frob}
M_F = \Pi^{-1} M \Pi =
\begin{pmatrix}
 D_1 & O & \dots & \dots & O \\
 A_{2,1} & D_2 & O  & & O \\
 A_{3,1} & A_{3,2} & D_3 & & \vdots \\
 \vdots & & & \ddots & \vdots \\
 A_{d,1} & \dots & & A_{d,d-1} & D_d
\end{pmatrix},
\end{equation}
where the $D_i$ are square diagonal blocks, which are either irreducible or zero matrices, $O$ are (rectangular) blocks of zeroes and the $A_{i,j}$ are rectangular non-negative matrices. The block $D_1$ represents the (unique) minimal subsystem of the rotated odometer.
Since $\chi(a)$ starts with $0$ for each $a \in \mathcal{A}$, the first column of $M$ is strictly positive, and so all $A_{i,1}$, $2 \leq i \leq d$, are non-zero.

Clearly
\begin{equation}\label{eq:Det}
 \det(M_F - \lambda I) = \prod_{k=1}^d \det(D_k - \lambda I_k),
\end{equation}
where $I$ and $I_k$ stand for the identity matrices of the correct size.
Hence every eigenvalue of $M$ must be an eigenvalue of $D_k$ for some $k$, and vice versa.
In particular, as a consequence of the Perron-Frobenius Theorem, the number of eigenvalues $\lambda_j$ of $M$
that are strictly greater than $1$ is at least the number of non-zero diagonal blocks in \eqref{eq:frob}.  It follows from \cite{BKMS2010} that the number of ergodic invariant measures for the rotated odometer $(I,F_\pi)$ corresponds 
to the number of eigenvalues greater than $1$ with non-negative eigenvectors.
Such eigenvalues are called {\em distinguished} and every non-zero block $D_k$ has a distinguished eigenvalue $\lambda_k > 1$ provided $A_{k,j}$ is zero for each $\lambda_j \geq \lambda_k$.

If a subsitution is not covering, then $M$ can still be put into the Frobenius form, and the same method of determining the number of ergodic invariant measures applies. The difference with the covering case is that the column sums no longer are equal to $2^N$.

\subsection{Doubling the middle symbol}\label{sec:doub} We now describe a modification of the stationary substitution $\chi: \mathcal{A} \to \mathcal{A}^*$ associated to a rotated odometer $F_\pi: I \to I$, which we need to study its skew-product extensions.

Recall from \eqref{skew-function} that the skew-function $\psi$ has two intervals of continuity, $[0, \frac{1}{2})$ and $[\frac{1}{2},1)$.
If the number $q$ in the definition of $F_\pi$, and, consequently, the number of sets in the partition $\mathcal P_{0,q}$ of the unit interval used to define an IET $R_\pi$ is odd, then the skew-function
$\psi$ is not constant on the middle interval $[\frac{m}{q}, \frac{m+1}{q})$.

For this reason, we split this middle interval into two halves, with symbols $m^+$
for $[\frac{m}{q}, \frac12)$ and $m^-$ for $[\frac12, \frac{m+1}{q})$,
so that the new alphabet $\mathcal{A}' = \{0, 1, \dots, m^+,m^-, \dots q-1\}$
has $q' = q+1$ letters.

The corresponding $q' \times q'$-matrix $M'$ has identical
$m^+$-th and $m^-$-th rows, but the $m^+$-th and $m^-$-th column need not be the same.
Consequently, compared to the eigenvalues of $M$,  $M'$ will have one extra eigenvalue $0$.
If $\chi$ is covering, the leading
eigenvalue is still $2^N$, but the corresponding left eigenvector
is $(1, \dots, 1, \frac12, \frac12,1,\dots, 1)$ and the column sum is $2^N$ only if you leave out the contribution of the $m^-$-th row, and also add the $m^+$-th and $m^-$-th columns together.

\begin{definition}\label{defn-weights-words}
Let $\mathcal A' = \{0,1,\ldots,m^+,m^-,\ldots, q-1\}$ be an extended alphabet, and $\chi: \mathcal A' \to (\mathcal A')^*$ be a substitution. Let $\psi: I \to \mathbb Z$ be the skew-function defined in \eqref{skew-function}. Then the \emph{weight} of a letter $\chi(a)$, $a \in \mathcal A'$, is given by
  \begin{align*} \psi(a) = \left\{ \begin{array}{ll} +1, & a \in \{0,\ldots, m-1, m^+\}, \\  -1, & a \in \{m^-,m+1,\ldots, q-1\}.\end{array} \right.
  \end{align*}
 The weight of a substitution word $\chi(a) = w_1 \cdots w_\ell$, for $\ell \in \mathbb N$, is given by
   $$\psi(\chi(a)) = \sum_{i = 1}^\ell \psi(w_i).$$ 
\end{definition}

Alternatively, note that the number of occurrences of each letter in the substitution word $ \chi(a)$, $a \in \mathcal A'$, is given by the $a$-th row $m_a$ of the associated matrix $M'$. Then
  $$\psi(\chi(a)) = m_a (\underbrace{1,\ldots,1}_{\textrm{from }0 \textrm{ to }m^-},\underbrace{-1,\ldots, -1}_{\textrm{from }m^+ \textrm{ to }q-1})^T.$$

\begin{remark} 
In the rest of the paper, when we consider the skew-product system \eqref{eq-skewprod}, we will often drop the prime accent and denote $\mathcal{A}'$ and $M'$ by $\mathcal{A}$ and $M$, keeping in mind that they are then an alphabet with $q+1$ letters, and a $(q+1) \times (q+1)$-matrix respectively.
\end{remark}

\begin{remark}
We will always order the eigenvalues by
absolute value:
$\lambda_0 \geq |\lambda_1|
\geq  \dots \geq |\lambda_{q}|$.
The corresponding left eigenvectors
are denoted by $\ell_j$, and scaled in such a way that for some Jordan decomposition
$M = UJU^{-1}$, the $\ell_j$'s are the rows of $U^{-1}$, and also $\ell_0$ is a non-negative vector.
\end{remark}

\subsection{Essential values}\label{sec-essential}

We start by recalling some basics of theory of essential values. Main references for this material are \cite{LM02,M04,S77}. 

We denote by $(X,\mathcal{B},\mu)$ the probability space, where $\mathcal{B}$ is the collection of Borel sets in $X$, and $\mu$ is the probability measure.

\begin{definition}\label{defn-cocycle}
Let $(X,\mathcal{B},\mu)$ be the probability space, and $F: X \to X$ an ergodic transformation.
A Borel map $\Psi: X \times \mathbb{Z} \to \mathbb{Z}$ is called a \emph{cocycle} for $F$ if 
 \begin{equation}\label{eq-cocycle}
    \Psi(x,n_1+n_2) = \Psi(F^{n_2}(x),n_1) + \Psi(x,n_2).
  \end{equation}
for every $n_1,n_2 \in \mathbb{Z}$ and every $x \in X$, and  
 $\mu \left(\bigcup_{n \in \mathbb{Z}} \left( \{x : F^n(x) = x\} \cap \{ \Psi(x,n) \ne 0\} \right) \right) = 0$.
\end{definition}

A cocycle $\Psi:X \to \mathbb{Z}$ is called a \emph{coboundary} if there exists a Borel map $\phi: X \to \mathbb{Z}$ such that
  \begin{equation}\label{eq-coboundary}
   \Psi(x,n) = \phi(F^n(x)) - \phi(x) \quad \text{ for all  } n \in \mathbb{Z}.
  \end{equation}

We now define essential values.

\begin{definition}\label{defn-essential}
Let $F: X \to X$ be an ergodic transformation on the probability space $(X,\mathcal{B},\mu)$. Let $\Psi: X \times \mathbb{Z} \to \mathbb{Z}$ be a cocycle for $F$. Then an element $e \in \mathbb{Z}$ is called an \emph{essential value} of $\Psi$ if, for every $A \in \mathcal{B}$ with $\mu(A)>0$,
  $$
  \mu\left( \bigcup_{n \in \mathbb{Z}}  A \cap F^{-n}(A) \cap \left \{ x : \Psi(x,n) = e \right\}  \right)>0.
  $$
\end{definition}
 
As described in \cite{LM02}, given any skew-function $\psi: X \to \mathbb{Z}$, one can define a cocycle $\Psi: X \times \mathbb{Z} \to \mathbb{Z}$ by
\begin{align}\label{eq-cocycle-psi}
\Psi(x,n) =
\begin{cases}
 \psi(F^{n-1}x)+ \cdots +\psi(F(x))+\psi(x), & n \geq 1 \\
 0, & n=0, \\
- \psi(F^{n}(x))  - \cdots - \psi(F^{-1}(x)), & n \leq -1.
\end{cases}
\end{align}
Conversely, given a cocycle $\Psi$, one can define a skew-function $\psi: X \to \mathbb{Z}$ by $\psi(x) = \Psi(x,1)$.

Denote the partial ergodic sums by $S_n \psi (x) = \sum_{i = 0}^{n-1} \psi(F^i(x))$. The following is well-known.

\begin{lemma}\label{lemma:ess}
Let $F: X \to X$ be an ergodic transformation on the probability space $(X,\mathcal{B},\mu)$. Let $\Psi: X \times \mathbb{Z} \to \mathbb{Z}$ be a cocycle for $F$, and let $\psi: X \to \mathbb{Z}$ be the skew-function defined by $\psi(x) = \Psi(x,1)$. Then the element $e \in \mathbb{Z}$ is an
{\em essential value} of the cocycle $\Psi$ if and only if for every positive measure $A \in \mathcal{B}$ there
exists an $n \in \mathbb{Z}$ such that
\begin{equation}\label{eq:essval}
\mu\left(A \cap F^{-n}(A) \cap \{ x \in X : S_n \psi(x) = e\}\right) > 0.
\end{equation}
\end{lemma}

In particular, $e = 0$ is always an essential value, by taking $n = 0$ in \eqref{eq:essval}.

\begin{definition}
Let $F: X \to X$ be an ergodic transformation on the probability space $(X,\mathcal{B},\mu)$, and let $\Psi: X \times \mathbb{Z} \to \mathbb{Z}$ be a cocycle for $F$.
We call the essential value $0$ \emph{non-trivial}, if \eqref{eq:essval} holds for $e = 0$ and some $n_A \neq 0$, for any $A \in \mathcal B$ with $\mu(A)>0$.

If $0$ is a non-trivial essential value, then the cocycle $\Psi: X \to \mathbb{Z}$ is \emph{recurrent}.

If $\Psi$ is not recurrent, then $\Psi$ is \emph{transient}.
\end{definition}

Let $\overline{\mathbb{Z}} = \mathbb{Z} \cup \{\infty\}$ be the one-point compactification of $\mathbb{Z}$. 

\begin{definition}\label{inf-essential}
We say that $\infty$ is an essential value for the cocycle $\Psi: X \times \mathbb{Z}\to \mathbb{Z}$ (with associated skew-function $\psi$) if for every $N \in \mathbb{N}$,
and every positive measure set $A \in \mathcal{B}$ there
exists an $n \in \mathbb{Z}$ such that
$$
\mu\left(A \cap F^{-n}(A) \cap \{ x \in X : |S_n \psi(x)| \geq N\} \right) >0.
$$
\end{definition}
Denote by $\overline E(\Psi)$ the set of essential values of the cocycle $\Psi$ in $\overline{\mathbb{Z}}$. One of the applications of essential values is to the study of \emph{skew-products} of dynamical systems.

\begin{definition}\label{defn-skew-product-details}
Let $(X,\mathcal B,\mu)$ be a probability measure space, and $F: X \to X$ be a measure-preserving transformation. Let $\nu$ be the counting measure on $\mathbb Z$. Consider a measure space $(X \times \mathbb{Z}, \mathcal B \times 2^{\mathbb Z}, \mu \otimes \nu)$, let $\psi: X \to \mathbb Z$ be a skew-function, and $\Psi: X \times \mathbb Z \to \mathbb Z$ be the associated cocycle, as in \eqref{eq-cocycle-psi}. 
Then the \emph{skew-product} of $F$ is the transformation
\begin{equation}\label{eq-skewprod-general}
T: X \times \mathbb{Z} \to X \times \mathbb{Z}, \qquad (x,n) \mapsto (F(x), n+\psi(x)),
\end{equation}
or, alternatively, denoting $T^m = T \circ \cdots \circ T$,
\begin{equation}\label{eq-skewprod-general-group}
T^m: X \times \mathbb{Z} \to X \times \mathbb{Z}, \qquad (x,n) \mapsto (F^m(x), n+\Psi(x,m)).
\end{equation}
\end{definition}

\begin{remark}
Formula \eqref{eq-skewprod-general-group} explicitly defines the skew-product of the action of any element $m$ of the acting group $\mathbb Z$, while formula \eqref{eq-skewprod-general} only specifies the skew-product of the action of a generator of $\mathbb Z$. Both definitions can be found in the literature. Formula \eqref{eq-skewprod-general-group} is suited to define skew-products of actions of finitely generated groups, which may have more than one generator. \end{remark}

We recall the following properties of essential values from \cite[Theorem 3.9, Propositions 3.14 and 3.15, Corollary 5.4]{S77}.

\begin{theorem}\label{thm-e-properties}
Let $F: X \to X$ be an ergodic transformation on the probability space $(X,\mathcal{B},\mu)$, and let $\Psi: X \times \mathbb{Z} \to \mathbb{Z}$ be a cocycle for $T$. Then the following is true:
\begin{enumerate}
\item $\overline E(\Psi)$ is a closed non-empty subset of $\overline{\mathbb{Z}}$, and $E(\Psi) = \overline E(\Psi) \cap \mathbb{Z}$ is a subgroup of $\mathbb{Z}$.
\item $E(\Psi) = \mathbb{Z}$ if and only if the skew-product of $F$ defined in Definition~\ref{defn-skew-product-details} is ergodic.
\item $\Psi$ is a coboundary if and only if $\overline E(\Psi) = \{0\}$, with $0$ being a non-trivial essential value.
\item $\Psi$ is transient if and only if $\overline E(\Psi) = \{0,\infty\}$, with $0$ being a trivial essential value.
\end{enumerate}
\end{theorem}

\subsection{Coboundaries for rotated odometers}

In the setting of rotated odometers, one easy sufficient condition for the cocycle $\Psi$ to be a coboundary is given in Lemma~\ref{lem:cob} below.

The weights of substitution words are defined in Definition~\ref{defn-weights-words}. In the lemma, we denote by $\mathcal A$ the alphabet with the doubled middle symbol, see Section~\ref{sec:doub}.

\begin{lemma}\label{lem:cob}
Let $(I,F_\pi)$ be a stationary covering rotated odometer with the associated substitution $\chi$. Let $\psi$ be the skew-function defined in \eqref{skew-function}, and let $\Psi: I \times \mathbb Z \to \mathbb Z$ be the associated cocycle. If $\psi(\chi(a)) = 0$ for each $a \in \mathcal{A}$, then $0$ is the only non-trivial essential value and $\Psi$ is a coboundary.
\end{lemma}

This result is well-known, cf.\ \cite{S77}. In the proof below, we just specify the coboundary.

\begin{proof}
As discussed in Section~\ref{sec:renorm}, the orbits of $(I,F_\pi)$ are coded by sequences of letters in $\mathcal A$, and there is the smallest subshift $(X_\chi,\sigma)$ such that every subword $W$ that it contains is a subword of an iterated substitution word.

 By \eqref{eq-coboundary}, we need to find $\phi$ such that the coboundary equation \eqref{eq-coboundary} holds. Let $L_1$ be a section as in Proposition~\ref{lem:periodicpermutations-1}, and $L_{1,a}$, $a \in \mathcal A$, be the sets in the partition of $L_k$.
 For $x \in L_{1,a}$, $a \in \mathcal A$, the itinerary of the orbit of $x$ for the first $m = |\chi(a)|$ steps is given by the substitution word. Let $\chi(a) = w_0 \dots w_{m-1}$, and set $\phi(x) = 0$ for all $x \in L_{1,a}$. Then for $1 \leq n \leq m$, all $x \in L_{1,a}$ and all $a \in \mathcal A$ define 
   \begin{align}\label{eq-coboundary-defn}\phi(F_\pi^i(x)) = \sum_{j = 0}^{n-1} \psi(F_\pi^j(x)) = \sum_{j = 0}^{n-1} \psi(w_j).\end{align}
 Since $\chi$ is covering, $\phi$ is now defined on $I$. Since $\psi(\chi(a)) = 0$ for all $a \in \mathcal A$, the coboundary equation extends for all $x \in I$ and all $n \in \mathbb Z$ by the same formula \eqref{eq-coboundary-defn}.
\end{proof}

\section{Discrepancy and diffusion coefficients}

In this section we compute the upper bounds on the discrepancy and the diffusion coefficient of a stationary covering rotated odometer $(I,F_\pi, \mbox{Leb})$.

\subsection{Discrepancy}\label{sec:discr}

Recall (see for instance \cite[Chapter 1, Sections 1 and 5]{KN}), that a sequence $(x_n)_{n \geq 1} \subset [0,1]$ is {\em uniformly distributed}
if
$$
\lim_{n\to\infty} \frac1n \#\{ 1 \leq j \leq n : x_j \in J\} = |J|
$$
for each interval $J \subset [0,1]$ and {\em well-distributed} if, for each interval $J \subset [0,1]$,
$$
\lim_{n\to\infty} \frac1n \#\{ k+1 \leq j \leq k+n : x_j \in J\} = |J|,
\quad \textrm{ uniformly in } k.
$$
By the Birkhoff ergodic theorem, if $(X,F,\mu)$ is ergodic, then for $\mu$-almost all $x \in X$ the orbit ${\rm orb}(x)$ is uniformly distributed. 

The {\em discrepancy} of $(x_n)_{n \geq 1}$
is a measure of how far the sequence deviates from the exact averages. More precisely:

\begin{definition}\label{def-discr}
Let $(x_n)_{n \geq 1} \subset [0,1]$ be a sequence of real numbers. Then the number
$$
 \mathfrak{D}_R = \sup_{0 \leq  a \leq b \leq 1} \left| \frac1R \#\{ 1 \leq j \leq R : x_j \in [a,b)\} - (b-a) \right|.
 $$
 is called the \emph{discrepancy} of $(x_n)_{n \geq 1}$.
 \end{definition}
 One can also define
 \begin{align}\label{eq-stardiscr}
 \mathfrak{D}^*_R = \sup_k \sup_{0 \leq  a \leq b \leq 1} \left| \frac1R \#\{ k+1 \leq j \leq k+R : x_j \in [a,b)\} - (b-a) \right|.
 \end{align}
 Clearly $\mathfrak{D}_R \leq \mathfrak{D}_R^*$. 
 Note that Definition~\ref{def-discr} does not require the sequence $(x_n)_{n \geq 1}$ to be uniformly (or well-) distributed.

A definition of discrepancy for substitution shifts, using cylinder sets instead of intervals, was used by Adamczewski \cite{Ad04}, in his study of discrepancies  of primitive substitutions. Adamczewski's results \cite{Ad04} were extended to $S$-adic transformations in \cite{BD14}. 
 For many rotated odometers the substitution matrices of the associated shifts are not primitive, therefore, we cannot use the results of \cite{Ad04,BD14}. In 
Theorem~\ref{thm:discr} below we obtain an estimate on the discrepancy of the sequence given by the orbit of $x \in I$ in the case when the associated substitution is covering, which means that Lebesgue measure is ergodic. Theorem~\ref{thm:discr1} below is a restatement of Theorem~\ref{thm:discr} of the introduction, added here for the convenience of the reader.

\begin{theorem}\label{thm:discr1}
Suppose that $F_\pi$ is covering and stationary,
 the associated matrix is diagonalizable, and the largest eigenvalue $\lambda_0$ has multiplicity $1$.
Then for Lebesgue-a.e.\ $x$, there is $C = C_x$ such that
the $F_\pi$-orbit of $x$ has discrepancy
$\mathfrak{D}_R \leq \mathfrak{D}^*_R \leq C_x \cdot R^{\gamma_0-1}$
for $\gamma_0 := \max\left\{ \frac{\log |\lambda_1|}{\log \lambda_0}, 0 \right\}$.
\end{theorem}

If $\lambda_0$ has algebraic multiplicity $\geq 2$, then the diffusion coefficient could be $1$. 
The assumption that $J$ is diagonalizable is for simplicity, mostly. However, if the Jordan block associated to $\lambda_1$ is
nontrivial, then we would need logarithmic correction 
factors as Adamczewski obtains 
for primitive substitution matrices, \cite{Ad04}.
We have no examples where $\lambda_1$ has a non-trivial
Jordan block, or $\lambda_0$ algebraic multiplicity $\geq 2$.

If $F_\pi$ is uniquely ergodic, then $C_x$ in the above theorem can be taken uniformly over all $x \in [0,1)$,
but no uniformity can be expected if there is a proper minimal subsystem, as for instance in Example~\ref{ex-5-3nue},
because there the $\mathbb{Z}$-extension of the minimal part is transient, with diffusion coefficient $1 > \gamma_0 =  0.694...$. The points in the minimal part are not Lebesgue typical, so there is no contradiction to Theorem~\ref{thm:discr1}.

\begin{proof} In this proof, we work with the rotated odometer $(I,F_\pi,\mbox{Leb})$, and so we do not have to split the middle interval of the partition $\mathcal P_{0,q}$ for $q$ odd.

Let $M = UJU^{-1}$ be the Jordan decomposition of $M$,
where $J$ is such that the eigenvalues on the diagonal
are in decreasing order, i.e., $\lambda_0 > |\lambda_1| \geq \cdots \geq |\lambda_{q-1}|$.
Then the first row $\ell_0$ of $U^{-1}$, which is the left eigenvector
associated to $\lambda_0 = 2^N$, is constant, because $\chi$ is covering.
Assume it is scaled so that $(1, \dots, 1)$ is this row.
Then for every symbol $a \in \mathcal{A} = \{0,\ldots, q-1\}$, the abelianization of
$\chi^n(a)$ is
\begin{align}\label{eq-vn}
V_n := (0, \dots,\underbrace{1}_{\text{position } a}, \dots, 0) UJ^nU^{-1}
= \underbrace{U_a}_{a\text{-th row}} J^n U^{-1}
= \sum_{j=0}^{q-1} U_{a,j} \lambda_j^n \ell_j,
\end{align}
where the last equality follows because $J$ is a diagonal matrix.
The length of the substitution word obtained from $a \in \mathcal{A}$ after $n$-th application of the substitution rule is
\begin{eqnarray}\label{eq:chin}
|\chi^n(a)| &=& \left\| \sum_{j=0}^{q-1} U_{a,j} \lambda_j^n \ell_j \right\|_1 \\
&=& \left\| U_{a,0} q2^{nN} \left(1 + \sum_{j=1}^{q-1} \frac{U_{a,j}}{U_{a,0}}\left( \frac{\lambda_j}{\lambda_0} \right)^n
\frac{ \ell_j }{q} \right)  \right\|_1 \sim \left| U_{a,0} q2^{nN} \right|, \nonumber
\end{eqnarray}
where at the last step we use that $\lambda_0 = 2^{N} > |\lambda_j|$, $1 \leq j \leq q-1$, and so all summands in the sum in the brackets are negligible for $n$ large.

 We now compute $\mathfrak{D}_R^*$, see \eqref{eq-stardiscr}, for the orbit of a Lebesgue typical point in $I$. 
 
 Recall that, for $k \geq 0$, we have the partition $\mathcal P_{kN,q}$ of the unit interval into half-open subintervals of length $(q2^{kN})^{-1}$, where $N$ is defined in \eqref{eq-defnN}. We call an interval $G$ a \emph{$q2^{kN}$-adic interval} if $G \in \mathcal P_{kN,q}$. Recall that for each $k \geq 1$,  $L_{k}$ is the union of the first $q$ subintervals in the partition $\mathcal P_{kN,q}$. Denote by $L_{k,a}$, $0 \leq a \leq q-1$, these subintervals.
 
 Recall from Proposition~\ref{lem:periodicpermutations-1} that, associated to the rotated odometer $F_\pi$, there is a sequence of substitutions $(\chi_j)_{j \geq 1}$ with alphabet $\mathcal{A}$. Namely, the sets $L_{k,a}$, $0 \leq a \leq q-1$, form a coding partition $\mathcal P_{kN,q}^{cod}$, which codes the orbits of the $q2^{(k+1)N}$-adic intervals in the partition $\mathcal P_{(k+1)N,q}^{cod}$. That is, one defines a substitution $\chi_{k+1}: \mathcal{A} \to \mathcal{A}^*$, where the substitution word $\chi_{k+1}(a)$ records the order in which the sets in $\mathcal P_{kN,q}^{cod}$ are visited by the orbit of $L_{k+1,a}$ under the first return map $F_{\pi,k}: L_k \to L_k$. For $i \geq 1$, the composition $\chi_{k+i} \circ \cdots \circ \chi_{k+1} (a)$ codes the orbit of $L_{k+i,a}$ with respect to $\mathcal P_{kN,q}$, and the orbit of $L_{k+i,a}$ with respect to $\mathcal P_{0,q}$ is coded by the composition $\chi_{k+i} \circ \cdots \circ \chi_1$ of substitutions.

The assumption of the theorem is that $(\chi_i)_{i \geq 1}$ is stationary, i.e., there is a substitution $\chi: \mathcal{A} \to \mathcal{A}^*$ such that $\chi_i = \chi$ for all $i \geq 1$. Then, given a $q2^{(k+n)N}$-adic interval $L_{k+n,a}$, $0 \leq a \leq q-1$, the coding of its orbit relative to the partition $\mathcal P_{kN,q}$ is given by the substitution word $\chi^n(a)$. 

We will compute \eqref{eq-stardiscr} in two steps: first for any $q2^{kN}$-adic interval $G$, and then for an interval of arbitrary length.

So let $G$ be a $q2^{kN}$-adic interval, let $n \geq 1$ and let 
$y \in L_{k+n,a}$,
$a \in \mathcal{A}$, be a Lebesgue typical point, i.e., the orbit ${\rm orb}(y)$ returns to $L_{k}$ infinitely often.
The easiest case is when $G \in \mathcal P_{kN,q}$, then the number of visits of the orbit of $y$ of length $R = |\chi^{k+n}(a)|$ to $G$ is bounded above by $|\chi^n(a)|$. Note that this bound is not sharp: indeed, $|\chi^n(a)|$ counts visits to any $q2^{kN}$-adic interval in $L_k$, not just $G$. Under assumptions of the theorem, this bound is true for \emph{any} $q2^{kN}$-adic interval $G$ in $\mathcal P_{kN,q}$: indeed, since $\chi$ is covering, there exists a unique $a \in \mathcal{A}$ such that $G = F_\pi^t(L_{k,a})$ for some $t \in \mathbb N$. By construction, the number $t$ is less than the first return time of $L_{k,a}$ to $L_k$, and $G$ appears in the orbit of $L_{k,a}$, before it returns to $L_k$, only once. Then ${\rm orb}(y)$ visits $L_{k,a}$ if and only if it visits $G$, which implies that the bound on the number of visits to $L_{k,a}$ also holds for $G$. Thus we obtain, for any $q2^{kN}$-adic interval $G$, the following estimate.

Denote by $\mathfrak{D}_R^*(y, G)$ an element of the set over which the supremum is taken in \eqref{eq-stardiscr}. Then we have
\begin{align}
 \mathfrak{D}_R^*(y,G) &:= \frac1R \#\{0 \leq j < R : F_\pi^j(y) \in G \} - |G|  \leq   \frac{|\chi^{n}(a)|}{|\chi^{k+n}(a)|} - |G| \nonumber
\\ \label{eq12}
&=  \frac{1}{|\chi^{k+n}(a)|} \left( U_{a,0} \lambda_0^n
+ \sum_{j=1}^{q-1} U_{a,j} |\lambda_j|^n \| \ell_j \|_1 \right) -
\frac{1}{q2^{kN}} \\
&= \left( \frac{ U_{a,0} \lambda_0^n
+ \sum_{j=1}^{q-1} U_{a,j} |\lambda_j|^n \| \ell_j \|_1}
{U_{a,0} \lambda_0^{k+n}\nonumber
+ \sum_{j=1}^{q-1} U_{a,j} |\lambda_j|^{k+n} \| \ell_j \|_1}\right)  -
\frac{1}{q2^{kN}} \\
&= \frac{U_{a,0} \lambda_0^n}{q U_{a,0}  \lambda_0^{k+n}}
\left( \frac{1+ \sum_{j=1}^{q-1} \frac{U_{a,j}}{U_{a,0}} \left(\frac{|\lambda_j|}{\lambda_0}\right)^n \| \ell_j \|_1 }
{1+ \sum_{j=1}^{q-1} \frac{U_{a,j}}{U_{a,0}} \left(\frac{|\lambda_j|}{\lambda_0}\right)^{k+n} \frac{\| \ell_j \|_1}{q} } \right)
 - \frac{1}{q2^{kN}} \nonumber \\
 &\sim \frac{1}{q 2^{kN}}
\left( 1 + \sum_{j=1}^{q-1} \frac{U_{a,j}}{U_{a,0}} \left(\frac{|\lambda_j|}{\lambda_0}\right)^n \| \ell_j \|_1
- \sum_{j=1}^{q-1} \frac{U_{a,j}}{U_{a,0}} \left(\frac{|\lambda_j|}{\lambda_0}\right)^{k+n} \frac{\| \ell_j \|_1}{q}  \right) \label{eq15}
 - \frac{1}{q2^{kN}}  \\
&\leq C |G| \left(\frac{|\lambda_1|}{\lambda_0}\right)^n
\leq C' \,|G| \, R^{\gamma_0 - 1} < C' R^{\gamma_0 -1}, \label{eq16}
\end{align}
for some $C, C' > 0$ independent of $a$.
Here we use \eqref{eq-vn} to obtain \eqref{eq12}, the approximation $\frac{1}{1+x} \sim 1-x$ for small $x$ and the fact that $\left(\frac{|\lambda_j|}{\lambda_0}\right)^n \to_n 0$ to obtain \eqref{eq15}, and the fact that $\lambda_1$ dominates the eigenvalues $\lambda_j$, $j = 2,\ldots,q-1$, in the absolute value, to obtain \eqref{eq16}.

Now consider the orbit of arbitrary length $R$. As $G$ is $q2^{kN}$-adic, it suffices to take
$R \geq \max_{a \in \mathcal{A}} |\chi^k(a)|$. 

Let $u = u_1\dots u_R$ be a subword of 
the infinite sequence in $X_\chi$ which codes the orbit of a Lebesgue typical point $y \in I$ as in Lemma~\ref{lemma-code-orbits}, i.e., $u_\ell \in \mathcal{A}$ for each $1 \leq \ell \leq R$. This corresponds to an orbit of length $R$ for a point $y' \in \operatorname{orb}(y)$. 
By abuse of notation, and to distinguish with the case of orbits of specific length above, denote
$\mathfrak{D}_R^*(u): = \frac1R \#\{ 1 \leq j \leq R : F_\pi(y) \in G\}
- |G|$.

Take $n \geq 0$ maximal such that there is a letter $a \in \mathcal{A}$ such that $\chi^{k+n}(a)$ is a subword of $u$.
Once this $n$ has been found, we can take
a block $B_0 \in \mathcal{A}^*$ maximal such that
$\chi^{k+n}(B_0)$ is a subword of $u$. 

Next find the $s \in \mathbb{N}$  minimal and blocks $B_{\pm 1}, B_{\pm 2}, \dots, B_{\pm s}\in \mathcal{A}^*$
maximal such that $u$ is a subword of
\begin{equation}\label{eq:v}
v := \chi^{k+n-s}(B_{-s}) 
\cdots \chi^{k+n-1}(B_{-1})
\chi^{k+n}(B_{0})
\chi^{k+n-1}(B_{1})
\cdots 
\chi^{k+n-s}(B_{s}).
\end{equation}
Here we allow the blocks $B_{\pm i}$, $1\leq i \leq s$, to be empty, for the sake of convenience of notation, so that we could have a formula symmetric in $s$. Also, we choose $B_{\pm i}$ recursively, i.e.,$B_{\pm i}$ can only be chosen after $B_{\pm (i-1)}$ maximal are found, for $1 \leq i \leq s$.

This procedure ensures that the lengths $b_i := |\chi^{k+n-|i|}(B_i)|$, $1 \leq |i| \leq s$, are bounded. Indeed, note that the lengths $|B_{\pm i}|$ are bounded above by $H = \max \{|\chi(a)| : a \in \mathcal{A}\}$, the length of the longest word of the substitution. Indeed, let $B_{-i}$ be the first non-empty block such that $\chi^{k+n-i}(B_{-i})$ is adjacent to $\chi^{k+n}(B_0)$ on the left. Then $\chi^{k+n-i}(B_{-i})$ is contained in a longer substitution word $\chi^{k+n}(a)$, for some $a \in \mathcal{A}$, and so the block $B_{-i}$ must be a proper subword of the substitution word $\chi^i(a)$. Moreover, if $c \in \mathcal{A}$ is the rightmost letter in the subword $\chi^{i-1}(a)$, then $B_{-i}$ is a proper subword of $\chi(c)$, since otherwise we would have that $\chi^{k+n-i+1}(c)$ is a subword of $u$ adjacent to $\chi^{k+n}(B_0)$ on the left. This means that either $B_{-i+1} = c$ (if $i \geq 2$), or the block $B_0$ can be increased by adding $c$ on the left. Both contradict the maximality of $B_{\pm i}$, for $0 \leq i \leq s$. The argument proceeds by induction on $i$. Next, since there is only a finite number of options for $B_i$, the lengths $b_i$ are bounded. Namely, $b_i < \max\{ |\chi^{k+n - |i| +1}(a)| : a \in \mathcal A\}$.

Then there is $C >0$ and $\beta \in (0,1)$, such that, independently of
$n$ and $s$, we have $b_i \leq C \beta^{|i|} b_0$
for all $-s \leq i \leq s$.
In particular,
\begin{equation*}
 |u| \leq |v| \leq \frac{2C}{1-\beta} b_0.
\end{equation*}

The estimate \eqref{eq16} applies to $b_i$, $-s \leq i \leq s$,  so there exists $C' \geq 1$ such that
$$
b_i \mathfrak{D}^*_{b_i}(\chi^{k+n-|i|}(B_i))
\leq C' b_i^{\gamma_0} \leq C' C^{\gamma_0} \beta^{\gamma_0 |i|} b_0^{\gamma_0}.
$$
Hence
$$
\mathfrak{D}^*_R(u) \leq \frac{1}{R} \sum_{i=-s}^s b_i \mathfrak{D}^*_{b_i}(\chi^{k+n-|i|}(B_i))
\leq \frac{1}{R} \frac{2C'C^{\gamma_0}}{1-\beta^{\gamma_0}} b_0^{\gamma_0}
\leq \frac{2C'C^{\gamma_0}}{1-\beta^{\gamma_0}}R^{\gamma_0-1},
$$
which proves the theorem for any $q2^{kN}$-adic interval, for $k \geq 1$.

Now, let $G$ be any interval in $I$, and let $R = |\chi^{k+n}(a)|$, for $n \geq 1$ and $a \in \mathcal{A}$. Note that $G$ can be represented as a (possibly infinite) union of $q2^{kN}$-adic intervals, i.e., there are non-intersecting intervals $K_1,K_2,\ldots$, such that $G = \bigcup_{t \geq 1} K_t$. Each $K_t$ has length $c_t = (q2^{k_tN})^{-1}$ for some $k_t \in \mathbb N$, and $|G| = \sum_{t \geq 1} c_t$. Then using \eqref{eq16}  we obtain
  \begin{align}\label{eq-arbitraryG}
  \mathfrak{D}_R^*(y,G) & = \frac{1}{R} \left(\{0 \leq j \leq R : F^j_\pi(y) \in G\} - R |G| \right) \nonumber \\ 
  & = \frac{1}{R}  \sum_{t \geq 1} \left( \{ 0 \leq j \leq R : F^j_\pi(y) \in K_t\} - R c_t \right) = \frac{1}{R} \sum_{t \geq 1} R \, \mathfrak{D}_R^*(y, K_t) \nonumber \\ 
  & \leq \sum_{t \geq 1} C' c_t R^{\gamma_0 -1}  = C' |G| R^{\gamma_0 - 1} \leq C' R^{\gamma_0 - 1},
  \end{align}
since $|G| \leq 1$. The proof for arbitrary $R >0$ proceeds similarly to the case of the $q2^{kN}$-adic interval $G$, with \eqref{eq-arbitraryG} instead of \eqref{eq16}.  This concludes the proof of the theorem.
\end{proof}

\subsection{Diffusion coefficient}\label{sec:diff}

If the cocycle $\Psi$ defined by the skew-function $\psi$ is a coboundary, then all orbits are bounded, and hence the diffusion coefficient is $0$. For the other case, we compute the upper bound on the diffusion coefficient in Theorem~\ref{thm-diffusion}. Recall that we assume that the largest eigenvalue $\lambda_0$ has multiplicity $1$, i.e., we assume that $\lambda_0 > |\lambda_1| \geq \cdots \geq |\lambda_q|$. We also assume for the second largest eigenvalue that $|\lambda_1|>1$. This holds, in particular, if the matrix $M$ of the substitution $\chi$ has at least two non-zero diagonal blocks in the Frobenius form, see Section~\ref{matrix-measures}. Theorem~\ref{thm:dif} below is a restatement of Theorem~\ref{thm-diffusion} of the introduction, added here for the convenience of the reader.

\begin{theorem}\label{thm:dif}
 If the substitution $\chi$ is covering and stationary, the associated matrix $M$ is diagonalizable, the largest eigenvalue $\lambda_0$ has multiplicity $1$, and the second largest eigenvalue satisfies $|\lambda_1| >1$, then
 for the corresponding $\mathbb{Z}$-extension, the diffusion coefficients satisfy 
$$
\gamma(x) \leq \gamma_0 := \frac{\log |\lambda_1|}{\log \lambda_0} \quad \text{ for Leb-a.e. } x.
$$
\end{theorem}

The assumptions that $M$ is diagonalizable and that $|\lambda| \neq 1$
are not essential. Without them, there may appear extra lorarithmic terms in the diffusion.

\begin{proof}
In this proof, we work with the skew-product dynamical system \eqref{eq-skewprod} over a rotated odometer $(I,F_\pi,\mbox{Leb})$. Therefore, if $q$ is odd, we split the middle state $m := \lfloor q/2\rfloor$ in two states $m^+$ and $m^-$, according to the value of the skew-function, see Section~\ref{sec:doub} for details.
The substitution and the associated matrix are still called $\chi$ and $M$. Recall that, as a consequence of the Perron-Frobenius Theorem, the left eigenvector $\ell_0 = (1,\ldots, 1, \frac{1}{2}, \frac{1}{2}, 1,\ldots, 1)$, where $\frac{1}{2}$ correspond to the split symbol.

 Let $\ell_j$, $j = 1, \dots, q-1$, be the left eigenvectors associated to eigenvalues $\lambda_j$. 
 
 We note that, in the definition of the diffusion coefficient \eqref{eq-diff-defn} by \eqref{eq-skewprod-general-group} we have for $x \in I$
   $$\xi(T^n_\pi(x,k)) - k = \xi (F^n_\pi(x),k+\Psi(x,n)) - k = \Psi(x,n) = \sum_{i = 0}^{n-1} \psi(F^i_\pi (x)),$$ 
   where $\Psi$ is the cocycle associated to the skew-function $\psi$. Using the coding of orbits, the piece $\{x, F_\pi(x), \dots , F_\pi^{n-1}(x)\}$ is determined by a
 substitution word $W = w_1\dots w_n$, and the corresponding displacement in the skew-product system is given by
 $$
 \psi(W)  := \sum_{i=1}^n \psi(w_i),
 $$
 where we compute the weights $\psi(w_i)$ of symbols as in Definition~\ref{defn-weights-words}.
 
Now we use a similar construction for $W$ as was applied in the proof of Theorem~\ref{thm:discr} to the word $u$. Namely, take $s \geq 0$ maximal such that there is a letter $a \in \mathcal{A}$ such that $\chi^{s}(a)$ is a subword of $W$.
Once such an $s$ has been found, we can take
a block $B_0 \in \mathcal{A}^*$ maximal such that
$\chi^{s}(B_0)$ is a subword of $W$. Next find the blocks $B_{\pm 1}, B_{\pm 2}, \dots, B_{\pm s}\in \mathcal{A}^*$, possibly empty, such that
\begin{equation}\label{eq:w}
W := B_{-s} \chi(B_{-s+1}) 
\cdots \chi^{s-1}(B_{-1})
\chi^{s}(B_{0})
\chi^{s-1}(B_{1})
\cdots 
\chi(B_{s-1})B_{s}.
\end{equation}

 Let $B_i = b_{i,1} \cdots b_{i,t_i}$ for $-s \leq i \leq s$, and note that 
 $t_i < H = \max \{|\chi(a)| : a \in \mathcal{A}\}$ by the argument in the paragraph below the paragraph containing formula \eqref{eq:v}.

 Let $\vec u_i = (u_{i,0}, \dots, u_{i,q})$ be the row vector with
 $u_{i,a} = \#\{ 1 \leq k \leq t_i : b_{i,k} = a\}$, $a \in \mathcal{A}$, i.e., the $a$-th components of $\vec u_i$ is the number of occurrences of the letter $a \in \mathcal{A}$ in the block $B_i$. Take coefficients $c_{i,j} \in \mathbb{R}$ such that
$\vec u_i = \sum_{j=0}^{q} c_{i,j} \ell_j$, and note that, since there is a finite number of subwords of length at most $H$ in the alphabet $\mathcal{A}$,  $E_j = \sup_{i} |c_{i,j}| < \infty$, $0 \leq j \leq q$, is independent of $W$. 

Given a vector $\vec u_i$, we can compute its weight as the inner product  $\cdot$ of $\vec u_i$ with the vector $(1, \dots, 1, -1 \dots, -1)$, which assigns weights to the states of the partition $\mathcal{P}_{0,q}$ (with the middle interval split into $2$ for $q$ odd), with the middle interval split in the case of odd $q$; in particular, we have 
 $$ \psi(\ell_0) = \ell_0 \cdot (1, \dots, 1, -1 \dots -1) = 0$$ 
 for the left eigenvector associated to the largest eigenvalue $\lambda_0$. Next, we estimate the absolute value of the weight of the piece of the orbit corresponding to each block $B_i$, that is,
 \begin{eqnarray*}
|\psi(\chi^{s-|i|}(B_i)) |&=& \left| \left( \vec u_i \, M^{s-|i|} \right)\cdot (1, \dots, 1, -1 \dots, -1) \right|\\
&=& \left| \left(\sum_{j=0}^{q} c_{i,j} \lambda_j^{s-|i|} \ell_j \right) \cdot (1, \dots, 1, -1 \dots -1) \right| \leq E |\lambda_1|^{s-|i|},
\end{eqnarray*}
because $|\lambda_1| \geq |\lambda_j|$ for $2 \leq j \leq q$, and $\psi(\ell_0) = 0$. Here the constant $E  \geq \sup \{E_j |\psi(\ell_j)| : 1\leq j \leq q\}$.

Then, using \eqref{eq:w} and by assumption $|\lambda_1| > 1$ we obtain
$$
|\psi(W)| \leq \sum_{i=-s}^s |\psi(\chi^{s-|i|}(B_i))|
\leq E \sum_{i=-s}^s  |\lambda_1|^{s-|i|}
\leq \frac{2E}{1-1/|\lambda_1|}|\lambda_1|^s.
$$
Note that for $|B_0| = \|\vec{u_0} \|_1= \| \sum_{j=0}^{q} c_{0,j} \ell_j \|_1$, a computation similar to \eqref{eq:chin} shows that 
 $$|\chi^s(B_0)| \sim |c_{0,0} q \lambda_0^s|,$$ 
 so, in particular, since $|\chi^s(B_0)| $ increases with $s$, $|c_{0,0}| \ne 0$. Since there is at most a finite number of words of length at most $H$, there exists a constant $E' > 0$ such that 
  $$|W| \geq |\chi^s (B_0)|\geq E' |c_{0,0}| \lambda_0^s.$$ 
 Taking into account that $\lambda_0^{\gamma_0} = |\lambda_1|$, we obtain
$$
|\psi(W)| \leq \frac{2E}{{E'}^{\gamma_0}}\frac{|\lambda_1|}{|\lambda_1|-1}  \frac{1}{|c_{0,0}|^{\gamma_0}} \ |W|^{\gamma_0},
$$
as claimed.
\end{proof}

\begin{remark}
If a substitution is not covering, but $\psi(\ell_\mu) = 0$, where $\mu$ is an ergodic measure, and $\ell_\mu$ is the left eigenvector associated with the largest (in absolute value) eigenvalue $\ell_\mu$, then the same proof applies. If a rotated odometer is not uniquely ergodic, then by the discussion in Section~\ref{matrix-measures}, the substitution $M$ can be put into the Frobenius form $M_F$, where the first diagonal block $D_1$ corresponds to the minimal subsystem. Let $\mu$ be the measure supported on the minimal subset of $(I,F_\pi)$, and let $\lambda_\mu$ and $\ell_\mu$ denote the leading eigenvalue of $D_1$ and its corresponding left eigenvector. If $\psi(\ell_\mu) = 0$, then a proof similar to that of Theorem~\ref{thm:dif} applies, and one obtains that the diffusion coefficient is bounded above by $\max \left\{\frac{\log |\lambda_\mu'|}{\log |\lambda_\mu|}\right\}$, where $\lambda_\mu'$ is the second largest eigenvalue of $D_1$. However, so far we haven't found any example where both $D_1 \neq M$ and $\psi(\ell_\mu) = 0$.
 Instead, in all our examples that are not uniquely ergodic,
 the minimal subsystem becomes transient on the skew-product,
 with diffusion coefficient $\gamma = 1$, because
 $$
0 <  \frac1C |\psi(\ell_\mu)| \leq \frac{|\xi(T_\pi^n(x,k)) - k|}{n}
 \leq C  |\psi(\ell_\mu)| < \infty
 \qquad \mu\text{-a.e.,}
 $$
 for $C = \frac{ \max\{ U_{a,\mu}\, :\, a \in \mathcal{A}\} }{ \min\{ U_{a,\mu}\, :\, U_{a,\mu} > 0\} }$, and where $\xi: I \times \mathbb{Z} \to \mathbb{Z}$ is the projection.
\end{remark}

\section{Recurrence and non-ergodicity of skew-products}\label{sec-essentialvalues}

In this section, we study the ergodic properties of
the skew-product \eqref{eq-skewprod}. Recall from Section~\ref{sec:rotated} that a rotated odometer $(I,F_\pi)$, where $\pi$ is a permutation of $q$ symbols, admits at most $q$ invariant ergodic measures, and Lebesgue measure $\mbox{Leb}$ on $I$ is ergodic if and only if $(I,F_\pi)$ is covering, that is, $F_\pi: I \to I$ has no periodic points. In this section, we concentrate mostly on the study of skew-products on $I \times \mathbb Z$ with respect to the invariant measure
 $\mbox{Leb} \otimes \nu$ (where $\nu$ is counting measure on $\mathbb{Z}$), using theory of essential values, see Section~\ref{sec-essential} or \cite{M04,S77}.

\subsection{Recurrence}\label{sec:recur}
We assume that a rotated odometer $(I,F_\pi)$ is covering, that is, $\mbox{Leb}$ is an ergodic invariant measure. The following result goes back at least to Atkinson \cite{At1976}, but is used also in 
\cite{HW, CR2019}.

\begin{theorem}\label{Atkinson}
 Let $(X,\mathcal{B},\mu,F)$ be an ergodic transformation on a probability space.
 Then the skew-function $\psi:X \to \mathbb{Z}$ has integral
 $\int_X \psi \, d\mu = 0$ if and only if the associated cocycle $\Psi$ is recurrent, i.e.,
 $\Psi(x,n) = \sum_{j=0}^{n-1} \psi \circ F^j(x) = 0$ infinitely often, for $\mu$-a.e., $x \in X$.
\end{theorem}

In comparison, for $\mathbb{Z}^2$-extensions, as, for instance, in the wind-tree model, Atkinson's result fails. Nonetheless,
Avila \& Hubert \cite{AH20} show that for typical (although not all) direction of the flow in a wind-tree model, the motion is recurrent.

As we are working with a $\mathbb{Z}$-extension, Atkinson's result is applicable and we obtain the following corollary.
Recall that $\nu$ denotes counting measure on $\mathbb{Z}$.

\begin{lemma}\label{lemma-lebrecurrent}
Let $(I,F_\pi, \mbox{Leb})$ be a rotated odometer, and let $\psi: I \to \mathbb{Z}$ be the skew-function defined in \eqref{skew-function}.
If $\chi$ is covering, then the skew-product $(I \times \mathbb{Z}, T_\pi, \mbox{Leb} \otimes \nu)$ is recurrent. 
\end{lemma}

If Lebesgue is not the only invariant measure,
i.e., there is a proper minimal subsystem
$(I_{\min}, F_\pi, \mu)$,
we obtain the result:

\begin{lemma}\label{lemma-lebrecurrent1}
Let $(I,F_\pi,\mu)$ be a rotated odometer restricted to its (uniquely ergodic) minimal set, and let $\psi: I \to \mathbb{Z}$ be the skew-function defined in \eqref{skew-function}.
Let $\ell$ denote the leading left eigenvector of the diagonal block $D_1$ in the Frobenius form \eqref{eq:frob} of $M$. If $D_1$ is diagonalizable, then the skew-product $(I \times \mathbb{Z}, T_\pi, \mu \otimes \nu)$ is recurrent if and only if $\psi(\ell) = 0$.
\end{lemma}

\begin{proof}
The minimal part of the odometer is
represented by upper left diagonal block $D_1$ in the Frobenius form.
Using the permutation $\Pi$ of \eqref{eq:frob} for the alphabet $\mathcal{A}$, we can assume that $D_1$ is the associated matrix of the substitution $\chi_1$ acting on the first $q_1$ letters, describing the minimal part of $F_\pi$. 
Then $D_1$ is a $q_1 \times q_1$-matrix, with
eigenvalues $\lambda_0 > |\lambda_1| \geq \dots \geq \lambda_{q_1-1}$, and corresponding left (generalized) eigenvectors $\ell_0, \dots \ell_{q_1-1}$, where we note that since $D_1$ is primitive, then by the Perron-Frobenius Theorem the multiplicity of $\lambda_0$ is $1$.
This substitution then acts on the alphabet $\mathcal{A}_1 = \{ 0, \dots, q_1-1\}$. 
Applying $\psi$ to \eqref{eq:chin},
we can write
$$
\psi(\chi_1^n(0)) = \sum_{j=0}^{q_1-1} \psi(\ell_j) U_{0,j} \lambda_j^n
\quad \text{ and } \quad
|\chi_1^n(0)| = \left\| \sum_{j=0}^{q_1-1} \ell_j U_{0,j} \lambda_j^n \right\|_1.
$$
By minimality, $U_{0,0} \neq 0$.
Since $\lambda_0 > |\lambda_j|$ for $1 \leq j < q_1$,
\begin{eqnarray}\label{eq:quot}
\lim_{n\to\infty}\frac{\psi(\chi^n(0))}{|\chi^n(0)|} 
&=&
\lim_{n\to\infty}\frac{\sum_{j=0}^{q_1-1} \psi(\ell_j) U_{0,j} \lambda_j^n }{\left\| \sum_{j=0}^{q_1-1} \ell_j U_{0,j} \lambda_j^n \right\|_1} \nonumber \\[2mm]
&=&
\lim_{n\to\infty} \frac{ \psi(\ell_j) \left( 1 + \sum_{j=1}^{q_1-1}  \frac{U_{0,j}}{U_{0,0}} \frac{\lambda_j^n}{\lambda_0^n} \right) }
{\| \ell_0 \|_1 \left( 1 +  \sum_{j=1}^{q_1-1} \frac{ \| \ell_j \|_1 }{\| \ell_0\|_1} \frac{U_{0,j}}{U_{0,0}} \frac{\lambda_j^n}{\lambda_0^n} \right) }
= \frac{\psi(\ell_0)}{\| \ell_0 \|_1}.
\end{eqnarray}
 By Atkinson's result,  $(I \times \mathbb{Z}, T_\pi, \mu \otimes \nu)$ is recurrent if and only if $\int \psi \, d\mu = 0$.
 But $\mu$ is the unique invariant probability measure of the minimal subsystem, so Oxtoby's Theorem 
\cite{Ox52}
 and finally \eqref{eq:quot} give 
 for the fixed point $\rho = \rho_0\rho_1\rho_2\dots = \lim_{n\to\infty} \chi^n(0)$:
 $$
 \int \psi \, d\mu = \lim_{N\to\infty} \frac1N \sum_{j=0}^{n-1} \psi(\rho_j) = 
 \lim_{n\to\infty} \frac1{|\chi^n(0)|} \sum_{j=0}^{|\chi^n(0)|-1} \psi(\rho_j) = \lim_{n\to\infty} \frac{\psi(\chi^n(0))}{|\chi^n(0)|} = \frac{\psi(\ell_0)}{\| \ell_0 \|_1}.
 $$
 Therefore  $(I \times \mathbb{Z}, T_\pi, \mu \otimes \nu)$ is recurrent if and only $\psi(\ell_0) = 0$.
\end{proof}

This lemma applies also if $\chi$ is not covering (so Lebesgue is not ergodic), such as in Examples~\ref{ex-7-1ncue} and~\ref{ex-5-2nue}.

\subsection{Non-ergodicity} \label{sec:non-erg}

In this section we establish two criteria for non-ergodicity of skew-products over rotated odometers, based on the properties of the associated substitutions. We recall from Theorem~\ref{thm-e-properties} or \cite[Corollary 5.4]{S77} that the skew-product \eqref{eq-skewprod} is ergodic if and only if  $E(\Psi) = \mathbb{Z}$.

Below we prove Theorem~\ref{thm-1.0-non-essential}, restating it in a more technical way in Theorem~\ref{thm-1-non-essential} for the convenience of the reader. In this proof, we work with the skew-product dynamical system \eqref{eq-skewprod} over a rotated odometer $(I,F_\pi,\mbox{Leb})$. Therefore, if $q$ is odd, we split the middle state $m := \lfloor q/2\rfloor$ in two states $m^+$ and $m^-$, according to the value of the skew-function, see Section~\ref{sec:doub} for details.
The substitution and the associated matrix are still called $\chi$ and $M$.

\begin{theorem}\label{thm-1-non-essential}
Let $F_\pi: [0,1) \to [0,1)$ be a rotated odometer without intervals of periodic points, and suppose the associated substitution $\chi$ is stationary, so Lebesgue measure is ergodic. Let $\Psi: [0,1) \times \mathbb{Z} \to \mathbb{Z}$ be the cocycle, and let $\boldsymbol{d} := \gcd\{ \psi(\chi(a)) : a \in \mathcal{A}\}$. Then the subgroup $E(\Psi)$ of essential values is contained in $\boldsymbol{d}\mathbb{Z}$. In particular, if $\boldsymbol{d} > 1$, then $1$ is not an essential value.
\end{theorem}

\begin{remark}
 Note that Theorem~\ref{thm-1-non-essential} does not imply that $E(\Psi)$ is equal to $\boldsymbol{d}\mathbb{Z}$; rather, it gives a necessary condition for an integer $e \in \mathbb{Z}$ to be an essential value of the cocycle $\Psi$. In particular, $E(\Psi)$ can still be a trivial subgroup of $\mathbb{Z}$, see Theorem~\ref{thm-uniqergod-pisot} and examples after the proof of that theorem.
\end{remark}

\begin{proof} Recall from \cite{S77} that $e \in \mathbb{Z}$ is an essential value of $\Psi$ if and only if for every Borel set $B \subset [0,1)$ with $\mbox{Leb}(B)>0$ we have
\begin{equation}\label{eq-essential}
  \mbox{Leb}\left(\bigcup_{n \in \mathbb{Z}} \left(B \cap F_\pi^{-n}(B) \cap \{x \in B: \Psi(x,n) = e\} \right)\right) > 0.
\end{equation}
We will show that there exists a set $B$ such that if \eqref{eq-essential} holds, then $e \in \boldsymbol{d}\mathbb{Z}$, and then the statement of the theorem follows.

Take $B = L_k$ for some $k \geq 1$, where $L_k$ is the $k$-th section. Recall that $F_\pi$ is invertible except on the orbit of $0$, and so  $\mbox{Leb}(B \cap F_\pi^{-n}(B)) >0$ for some $n >0$ if and only if $\mbox{Leb}(B \cap F_\pi^{n}(B)) >0$.

The section $B = L_k$ of length $2^{-kN}$ is subdivided into the intervals  $L_{k+1,a}$, $0 \leq a \leq q-1$, of length $q^{-1}2^{-kN}$. The length of the substitution word $|\chi(a)|$ determines the first return time of $L_{k+1,a}$ to $L_k$ under iterations of the rotated odometer map $F_\pi$, that is, we have 
 \begin{align*}
 F_{\pi}^{|\chi(a)|}(L_{k+1,a}) \subset L_k, \quad \textrm{ and } \quad F_{\pi}^{s}(L_{k+1,a}) \cap L_k= \emptyset \quad \textrm{ for } 0 < s < |\chi(a)|.
 \end{align*}
We claim that the intersection 
 $$
 S_{k+1,a,b} = F_{\pi}^{|\chi(a)|}(L_{k+1,a}) \cap L_{k+1,b}, \quad b \in \mathcal A,
 $$
 is either empty, or is at most a finite union of half-open intervals.  To see that recall that $F_\pi = \mathfrak{a} \circ R_\pi$, where $R_\pi$ is a finite IET, and $\mathfrak{a}$ is the von Neumann-Kakutani map. The restriction of each $F_\pi^i$, $0 \leq i < |\chi(a)|$, to $L_{k,a}$ is a translation; the interval $L_{k,a}$ is re-arranged under $F^i_\pi$ (by the von Neumann-Kakutani part of the composition) only after its orbit reaches the set $R_\pi^{-1}\left( [ 1 - 2^{-kN}, 1 )\right)$, which happens for $i = |\chi(a)| - 1$. We then have $F^{|\chi(a)|}_\pi(L_{k+1},a) \subset L_k$. 
 
 Recall from \eqref{eq-odometer} that $I$ is subdivided into intervals $I_\ell$, $\ell \geq 1$, of length $|I_\ell| = 2^{-\ell}$, on which $\mathfrak{a}$ is continuous, and which are re-arranged under $\mathfrak{a}$ in the opposite order. Since $q \ne 2^r$, $r \in \mathbb N$, the image $R_\pi \circ F^{|\chi(a)-1|}_\pi(L_{k+1,a})$ may be contained in the union of more than one interval $I_\ell$, and there is a single letter $\iota \in \mathcal A$ such that $R_\pi \circ F^{|\chi(\iota)|-1}_\pi(L_{k+1,\iota}) \supset \bigcup_{\ell > \ell_\iota} I_\ell$, where $\ell_\iota = 2^{kN+1}$. Then $\mathfrak{a} (I_\ell) \subset L_{k+1,0}$ for $\ell > \ell_\iota$, and, for $\ell \leq \ell_\iota$, $I_\ell$ is mapped into a finite union of the sets $L_{k+1,b}$, $b \in \mathcal A$. It follows that either $S_{k+1,a,b} = \emptyset$, or $S_{k+1,a,b}$ is a finite union of half-open intervals. 

Next, note that $L_k \cap F^n_\pi(L_k)$ is non-empty if and only if one of $L_{k,a}$, $a \in \mathcal A$, is mapped into $L_k$ by a power of $F_\pi$. That is, $L_k \cap F^n_\pi(L_k)$ is non-empty if and only if $n = \sum_{i = 1}^t |\chi (a_i)|$ for some $t \geq 1$, and $a_i \in \mathcal A$, $1 \leq i \leq t$. Let $D \subset L_k$ be a subinterval for which $F^n_\pi(D) \subset L_k$; then the itinerary of every point $x \in D$ is described by the concatenated substitution word $W = \chi(a_1) \cdots \chi(a_t)$, and for this word we have
  $$
  \psi(W) = \psi(\chi(a_1))+ \cdots + \psi(\chi(a_t)).$$
Since every weight $\psi(\chi(a_i))$ is a multiple of $\boldsymbol{d}$, for all $x \in D_r$ we have $\Psi(x,n)  \in \boldsymbol{d}\mathbb{Z}$. Since this argument is true for any $n \in \mathbb{Z}$ such that $\mbox{Leb}(L_k \cap F^{-n}(L_k)) >0$, we have that $E(\Psi) \subset \boldsymbol{d}\mathbb{Z}$.
\end{proof}

\begin{example}
{\rm
By Theorem~\ref{thm-1-non-essential} the skew-products over the following  rotated odometers are not ergodic:

\begin{enumerate}
\item For $q=3$ and $\pi = (012), (021)$ in Examples~\ref{ex-3-1ue} and \ref{ex-3-2ue}, and for $q=9$ and $\pi=(1,7,4)(2,5)(3,6)$ in Example~\ref{ex-9-2ue} the rotated odometers are uniquely ergodic with Lebesgue invariant measure. In these examples $\boldsymbol{d} :=  \gcd\{ \psi(\chi(a)) : a \in \mathcal{A}\} = 2$, so $E(\Psi) \subset 2 \mathbb{Z} $ and the skew-product $(I \times \mathbb{Z}, T_\pi, \mbox{Leb} \otimes \nu)$ is not ergodic.

\item For $q=11$ in Example~\ref{ex-q11}, the rotated odometer determined by the permutation $\pi=(0,2,7,6,5,4,3,8,10,1,9)$ is not uniquely ergodic. 
Here again $\boldsymbol{d} =  \gcd\{ \psi(\chi(a)) : a \in \mathcal{A}\} = 2$, so $E(\Psi) \subset 2 \mathbb{Z} $ and the skew-product $(I \times \mathbb{Z}, T_\pi, \mbox{Leb} \otimes \nu)$ is not ergodic.
\end{enumerate}
}
\end{example}

We now prove Theorem~\ref{thm-uniqergod-pisot}.
Recall that a Pisot number is an algebraic integer $> 1$ such that its Galois conjugates are all strictly inside the unit circle.
If $\lambda_j$ is Pisot, this does not preclude that there are other eigenvalues outside the unit disc, even if the weight of the corresponding left eigenvector is non-zero.
After all, the characteristic polynomial of $M$ need not be irreducible.

In the proof below $M$ is a $(q +1) \times (q+1)$-matrix, i.e., $M$ corresponds to the substitution $\chi$ with doubled middle symbol, see \eqref{sec:doub}. The alphabet in Theorem~\ref{thm-uniqergod-pisot} is given by $\mathcal{A} = \{0,1\ldots, m^-, m^+, \dots, q-1\}$, i.e., $\#\mathcal{A} = q+1$, and the numbering in the sums below is from $0$ to $q$.

\begin{proof}[Proof of Theorem~\ref{thm-uniqergod-pisot}]
We will prove that under the assumptions of the theorem, the essential values of the system are $\overline{E}(\Psi) = \{ 0 ,\infty\}$, with $0$ a non-trivial essential value.

Let $t \geq 1$ be the algebraic and geometric multiplicity of $\lambda_c$. By assumption, $\lambda_c$ has $t$ linearly independent left eigenvectors $\ell_{c},\ldots,\ell_{c+t-1}$.  For the smallest $c' > c$ with $\psi(\ell_{c'}) \neq 0$,
 we have by assumption that $|\lambda_{c'}| < 1 < \lambda_c$.
Apply the weight function $\psi$ to \eqref{eq:chin}
to obtain
$$
 \psi(\chi^n(b)) = \sum_{a = 0}^q \psi(\ell_c) U_{b,a}(\lambda_a^n + \epsilon_a n \lambda_a^{n-1} ) =  \left( \sum_{j=0}^{t-1} \psi(\ell_{c+j}) U_{b,c+j} \right) \lambda_c^n + R(n),
$$ 
 where $U_{b,a}$ is the $(b,a)$-entry in the matrix $U$ of some Jordan block decomposition $M = U J U^{-1}$, and $\epsilon_a = 1$ if the Jordan block of $\lambda_a$ is non-diagonal, and $\epsilon_a = 0$ otherwise. Since $\lambda_c$ is the only eigenvalue of absolute value $\geq 1$, whose left eigenvector has non-zero weight, only that term remains
in the rightmost expression, and the remainder term $R(n)$ is exponentially small. In particular, since  $\lambda_a < \lambda_{c'} < 1 < \lambda_c$ for all $a \geq c'$, there exists $n_0$ such that for all $n> n_0$ and all such $a$ we have $n \lambda_a^{n-1} < \lambda_{c}^n$.

Now $\sum_{j=0}^{t-1} \psi(\ell_{c+j}) U_{b,c+j} \neq 0$ for at least one $b$, because otherwise the columns $c, \dots, c+t-1$ of $U$ are linearly dependent, contradicting that $U$ is invertible.
This implies that $|\psi(\chi^n(b))| \to_{n \to \infty} \infty$,
and some $(x,k) \in I \times \mathbb{Z}$ has an unbounded $T_\pi$-orbit.  This shows that $\psi$ is not a coboundary and $\infty \in \overline E(\Psi)$. Since $\chi$ is covering, Lemma~\ref{lemma-lebrecurrent} shows that $0 \in \overline E(\Psi)$. It remains to show that $\overline E(\Psi)$ contains no integer $e \neq 0$.

 Let $H = \max_{a \in \mathcal{A}} |\chi(a)|$, and $E = \sup\{|U_{b,a} \psi(\ell_a) |: b,a \in \mathcal A \}$.
The eigenvalue $\lambda_c$ is algebraic of degree $d \leq q$,
so the corresponding eigenvectors can be chosen in
the field extension $\mathbb{Q}(\lambda_c)$, and we can write, for each $0 \leq j \leq t-1$,
\begin{align}\label{eq-Uac}
\psi(\ell_{c+j}) U_{a,c+j} = \frac{1}{\Delta_{c+j}} \sum_{k=0}^{d-1} p_{k,a}^j \lambda_c^k
\end{align}
for integers $\Delta_{c+j} \ne 0$ and $p_{k,a}^j$.  Let 
$$
\Delta_c =  \prod_{j=0}^{t-1} \Delta_{c+j},
\qquad\quad \Delta_{c \div j} =  \frac{\Delta_c}{\Delta_{c+j}}\ \text{ for } 0\leq j \leq t-1.
$$
Set $G_j = 2\sum_{a \in \mathcal{A}} \sum_{k=0}^{d-1} |p_{k,a}^j| \in \mathbb{N}$, and  
  $$G = t \, \Delta_c \ \max \{G_j: 0 \leq j \leq t-1\}.$$ 
  
Since $\lambda_c$ is Pisot, by the Garsia Separation Lemma \cite[Lemma 1.51]{G62}, there is $C = C(G) > 0$ such that
for all $\alpha \in \mathbb{N}$ and all $\varepsilon_j \in \{-G, \dots, G\}$, $0 \leq j \leq \alpha$,
we have
\begin{equation}\label{eq:GLS}
\left| \sum_{j=0}^\alpha \varepsilon_j \lambda_c^{-j} \right|
 \geq C \lambda_c^{-\alpha} \quad \text{ or } \quad
 \left| \sum_{j=0}^\alpha \varepsilon_j \lambda_c^{-j} \right| = 0.
\end{equation}
Take $k_0 \in \mathbb{N}$ so large that $k_0 > n_0$, and
\begin{equation}\label{eq-k0choice}
|\lambda_c|^{k_0} > \frac{2\Delta_c |e|}{C} \quad  \text{ and } \quad |\lambda_{c'}|^{k_0}  < \frac{1-|\lambda_{c'}|}{4 H E q}.
\end{equation} 
Thus, for a fixed $F_\pi$, the constant $k_0$ depends only on the choice of $e \in \mathbb{Z}$.

Now let $A \subset [0,1]$ be an arbitrary set of positive measure. To show that $e \neq 0$ is an essential value, for every such $A$ we need to find an $n$ such that
\begin{align}\label{eq-findn}
\mbox{Leb}\left(A \cap F_{\pi}^{-n}(A) \cap \left\{ x \in A :\sum_{m=0}^{n-1} \psi \circ F_{\pi}^m(x)  = e \right\} \right) > 0.
\end{align}
In fact, it is not a restriction to assume that $A$ is a $q2^{kN}$-adic interval for some $k \in \mathbb{N}$, since every $A$ would contain such an interval, for $k$ large enough. We will prove that it is possible to choose $A$ so that \eqref{eq-findn} does not hold for any $n \in \mathbb N$. This will show that $e$ cannot be an essential value.

Since $F_\pi$ is covering, there is a $q2^{kN}$-adic subinterval $A'$ of $L_{k}$ such that $F^{k'}_\pi$ maps $A'$ continuously onto $A$ for some $k' \geq 0$. Pulling $A$ back by $F_\pi^{-k'}$, \eqref{eq-findn} is equivalent to finding an $n$ such that 
$$
\mbox{Leb}\left(A' \cap F_{\pi}^{-n}(A') \cap \left\{ x \in A' : \sum_{m=0}^{n-1} \psi \circ F_{\pi}^m (x)= e \right\} \right) > 0.
$$
The return times to the set $L_{k}$ satisfy
$n = \sum_{i=1}^r |\chi^{k}(a_i)|$ for some $r \geq 1$ and $a_i \in \mathcal{A}$. We now assume that $k$ is large enough (and, consequently, $A$ is small enough) so that $ \min \{|\chi^k(a)| : a \in \mathcal{A}\} > k_0$, where $k_0$ is the constant we chose in \eqref{eq-k0choice}.

Since the rotated odometer $(I,F_\pi, \mbox{Leb})$ is ergodic, $A$ contains a full measure set of points returning to $A$ infinitely often. Let $y$ be a Lebesgue typical point in $A$, and suppose $F^n(y) \in A$ for some $n$. By the choice of $A$ we have $n >k_0$, and by the coding procedure in Proposition~\ref{lem:periodicpermutations-1}, the itinerary of $y$ is coded by a word $W$ in the alphabet $\mathcal{A}$ of length $n$. Since $F_\pi$ acts by interval exchanges, there is a positive measure set $A_W \subset A$ (in fact, a subinterval) of points with the same itinerary as $y$.

Similarly to what we did in the proofs of Theorems~\ref{thm-diffusion} and \ref{thm:discr}, we can find (possibly empty) maximal blocks $B_i = b_{i,1} \dots b_{i,h_i}$, $-s \leq i \leq s$, of length $h_i \leq H$ such that
\begin{eqnarray*}
W & = &
\chi^{k_0}(B_{-s}) \chi^{k_0+1}(B_{-s+1}) \cdots
\chi^{k_0+s-1}(B_{-1}) \chi^{k_0+s}(B_0) \chi^{k_0+s-1}(B_1) \\
& & \cdots  \chi^{k_0+1}(B_{s-1}) \chi^{k_0}(B_s).
\end{eqnarray*}

Thus, for all $x \in A_W$,
\begin{align*}
\sum_{m=0}^{n-1} \psi \circ F_{\pi}^m(x) =
\sum_{i=-s}^s \sum_{j=1}^{h_i} \sum_{a = c}^q \psi(\ell_a) U_{b_{i,j},a}
 (\lambda_a^{k_0+s-|i|} + \epsilon_a n \lambda_a^{k_0 + s - |i|-1}) = S_1 + S_2
 \end{align*}
 where
 \begin{eqnarray*}
 |S_1| &:=&  \left| \sum_{i=-s}^s \sum_{j=1}^{h_i} \sum_{a = c'}^q \psi(\ell_{a}) U_{b_{i,j},a}  \left(\lambda_a^{k_0+s-|i|} + \epsilon_a  n \lambda_a^{k_0 + s - |i|-1}\right) \right| \\
 &\leq& 2 H E q \sum_{i=-s}^s |\lambda_{c'}|^{k_0+s-|i|}
 \leq \frac{4H E q  |\lambda_{c'}|^{k_0}}{1- |\lambda_{c'}|} < 1, 
\end{eqnarray*}
because $|\lambda_a| \leq |\lambda_{c'}| < 1$ for all $a < c'$,
and
\begin{eqnarray} \label{eq-S2t}
 S_2 & := &  \lambda_c^{k_0+s}  \sum_{i=-s}^s \sum_{j=1}^{h_i} \sum_{z=0}^{t-1} \psi(\ell_{c+z}) U_{b_{i,j},c+z} \, \lambda_c^{-|i|}  \nonumber \\ 
 &=& \lambda_c^{k_0+s} \sum_{i=-s}^s \sum_{j=1}^{h_i}  \left( \sum_{z = 0}^{t-1} \frac{1}{\Delta_{c+z}}  \sum_{k=0}^{d-1} p^z_{k,b_{i,j}} \lambda_c^{k-|i|} \right) \\ \nonumber
 &=& \frac{\lambda_c^{k_0+s+d}}{\Delta_c} \sum_{u=1}^{s+d} \underbrace{\sum_{z = 0}^{t-1} \Delta_{c \div z}\sum_{k=0}^{d-1}
 \left( \sum_{j=1}^{h_{u+k-d}} p^z_{k,b_{u+k-d,j}} +
 \sum_{j=1}^{h_{-(u+k-d)}} p^z_{k,b_{-(u+k-d),j}} \right) }_{\text{integers in } \{-G, \dots, G \}} \lambda_c^{-u}  \\[2mm] \nonumber
&= & 0
\quad \text{ or } \quad |S_2| \geq \frac{|\lambda_c|^{k_0+s+d}}{\Delta_c}  C |\lambda_c|^{-s-d} \geq |\lambda_c|^{k_0} \frac{C}{\Delta_c}, 
\end{eqnarray}
depending on which possibility of the Garsia Separation Lemma holds.
In the first case,  $|S_1+S_2| < 1$, and in the second case, $|S_1+S_2| > -1 + 2|e| \geq |e|$. Thus we found a set for which \eqref{eq-findn} fails, which shows that $e \neq 0$ is not an essential value.
\end{proof}

\begin{example}
This theorem applies to permutations $(012)$ and $(021)$
for $q = 3$ in Examples~\ref{ex-3-1ue} and \ref{ex-3-2ue}, thanks to the fact that in these examples, all left eigenvectors to eigenvalues $1$ have weight $\psi(\ell_j) = 0$. Why this latter property holds for all our examples is not clear to us, but see Lemma~\ref{lem:weight} for some discussion.
A more interesting example is $(1,7,4)(2,5)(3,6)$ in Example~\ref{ex-9-2ue} because it has the second largest eigenvalue of multiplicity $2$. The rotated odometer corresponding to the substitution $(02431)$ for $q=5$ in Example~\ref{ex-5-3nue} has two ergodic invariant measures, and a non-integer Pisot eigenvalue. In all these examples Theorem~\ref{thm-uniqergod-pisot} applies, and so $\overline{E}(\Psi) = \{0,\infty\}$, with non-trivial essential value $0$.
\end{example}

\subsection{Left eigenvectors of zero weight.}
The above theorem requires, among other assumptions, that the eigenvalues $|\lambda_j| = 1$ have weight $\psi(\ell_j) = 0$. This looks an artificial assumption, but all the examples that we know (including the ones in the next section) have this property. As a partial explanation for this, we consider in the next lemma some situations where $\lambda_j \in \{0,-1,1\}$ and $\psi(\ell_j) = 0$.
The proof is straightforward.

Let $|\chi(a)|_b$ be the number of occurrences of letter $b$ in $\chi(a)$. We denote by $1_a$ the vector with $1$ in the $a$-th entry, and $0$ in the other entries.

\begin{lemma}\label{lem:weight}
 Let $\chi$ be a stationary substitution, and $M$ the associated matrix.
 \begin{enumerate}
  \item If $\chi(a) = \chi(b)$ for some $a < b$,
  then $1_a - 1_b$ is a left eigenvector to eigenvalue $0$.
   \item If $|\chi(a)|_a = |\chi(b)|_a+1$ and
   $|\chi(a)|_b = |\chi(b)|_b-1$
   for some $a < b$, and $|\chi(a)|_j = |\chi(b)|_j$ for $j \neq a,b$,
  then $1_a - 1_b$ is a left eigenvector to eigenvalue $1$.
     \item If $|\chi(a)|_a = |\chi(b)|_a-1$ and
   $|\chi(a)|_b = |\chi(b)|_b+1$
   for some $a < b$, and $|\chi(a)|_j = |\chi(b)|_j$ for $j \neq a,b$,
  then $1_a - 1_b$ is a left eigenvector to eigenvalue $-1$.
     \item If $|\chi(a)|_a = |\chi(b)|_a+1$ and
   $|\chi(a)|_b = |\chi(b)|_b-1$
   for some $a < b$ (like in (2)),
   $\chi(c) = \chi(d)$,
   $|\chi(a)|_c = |\chi(b)|_c+1$ and
   $|\chi(a)|_d = |\chi(b)|_d-1$
   and $|\chi(a)|_j = |\chi(b)|_j$ for all other indices $j$,
  then $1_a - 1_b + 1_c - 1_d$ is a left eigenvector to eigenvalue $1$.
 \end{enumerate}
If additionally $0 \leq a < b \leq m^+$ or $m^- \leq a < b\leq q$, then the left eigenvectors mentioned in (1)-(3) have weight $0$. A similar extended statement holds for case (4).
\end{lemma}

The conditions in this lemma may look cumbersome and arbitrary, but they do occur because of certain structure of the (stationary) rotated odometer $F_\pi$ and its renormalization.
Indeed, if $L_{1,a}$ and $L_{1,b}$ are $q2^N$-adic subintervals of $L_1$ that have the same return time $n = |\chi(a)| = |\chi(b)|$ to $L_1$, and $F_{\pi}^j$ is an isometry on $L_{1,a} \cup L_{1,b}$,
then if $\frac{a+1}{q}$ lies between $F_\pi^j(L_{1,a})$ and $F_\pi^j(L_{1,b})$ for some $j$,
and no other $\frac{a'}{q}$, $a' \in \mathcal{A}$, has this property, 
then the condition of part (2) holds.
Naturally, this is easiest if $b = a+1$, and it applies to Examples 5.10, 5.11, 5.12, 5.13 (with $a=0, b=4$)

A similar phenomenon occurs in Example 5.5 with $a = 0, b=3$, leading to case (3) of the lemma.

\section{Examples}\label{sec:exam}

In the tables below we list the properties of several stationary (and one non-stationary)  substitutions, whether they are covering or not, and if covering, then whether they are uniquely ergodic or not.

We give the characteristic polynomial
and eigenvalues of the associated matrix $M$ that emerges after doubling the middle symbol $m$ to $m^+$ and $m^-$, see Section~\ref{sec:doub} . The number of ergodic measures of the rotated odometers can easily be read off from the Frobenius form of $M$, see Section~\ref{matrix-measures}. The quantities $\gamma_0$ and $\overline{E}(\Psi)$ are the Lebesgue typical diffusion coefficient and the set of essential values.

In the last line we give the weights of those left eigenvectors of $M$ that belong to non-zero eigenvalues, as
weights of eigenvectors of the eigenvalue $0$ don't affect anything. The scaling of the eigenvectors is arbitrary, as given by a program in Sage which we used to compute the examples. For proofs of our theorems it only matters if the weights are zero or non-zero. Non-rational values are marked by $*$, and we only give them up to at most two first digits. In the case of irrational eigenvalues it is usually clear which (quadratic or cubic) polynomial factors they are a root of. 
Values which are not reported are marked by $**$.

\subsection{Uniquely ergodic rotated odometers}

\begin{example}\label{ex-3-1ue}
{\rm
$$
\begin{array}{l|l}
\hline\hline & \\[-3mm]
\begin{array}{l}
 \pi = (012), \qquad q = 3 \\[2mm]
\textrm{covering}\\[2mm]
\#\{\textrm{ergodic measures}\} = 1 \\[2mm]
\gamma_0 = \frac12, \quad
\overline{E}(\Psi) = \{ 0, \infty\}

\end{array} & M = \begin{pmatrix}
   1 & 0 & 1 & 2 \\
   1 & 0 & 1 & 2 \\
   1 & 0 & 1 & 2 \\
   2 & 2 & 0 & 0
  \end{pmatrix}\\ & \\

 \hline & \\[-3mm]
\begin{array}{l}
a \in \mathcal{A} \\ \textrm{Weight } \psi (\chi (a))
\end{array} & \begin{array}{c |c |c |c } 0 & 1^+ & 1^- & 2    \\ \hline  -2 & -2 & -2&4 \end{array} \\ & \\[-3mm] \hline & \\[-3mm]
\textrm{\ Characteristic polynomial} &
 (x-4)(x+2)x^2\\[2mm] \hline & \\[-3mm]
\begin{array}{l}
\textrm{Eigenvalues} \\ 
\textrm{Weights of left eigenvectors}
\end{array} &
\begin{array}{ c |c |c }  4 & -2 & 0 (\times 2) \\ \hline  0 & -1 &   \phantom{}^{**} \end{array}
\\[-3mm]

\\[2mm] \hline \hline
\end{array}
$$
}
\end{example}

\begin{example}\label{ex-3-2ue}
{\rm
$$
\begin{array}{l|l}
\hline\hline & \\[-3mm]
\begin{array}{l}
 \pi = (021), \quad q = 3 \\[2mm]
\textrm{covering}\\[2mm]
\#\{\textrm{ergodic measures}\} = 1 \\[2mm]
\gamma_0 = \frac12, \quad  \overline{E}(\Psi) = \{ 0, \infty\}

\end{array} & M = \begin{pmatrix}
   2 & 2 & 2 & 4 \\
   1 & 0 & 0 & 0 \\
   1 & 0 & 0 & 0 \\
   1 & 0 & 0 & 0
  \end{pmatrix}\\ & \\

 \hline & \\[-3mm]
\begin{array}{l}
a \in \mathcal{A} \\ \textrm{Weight } \psi (\chi (a))
\end{array} & \begin{array}{c |c |c |c } 0 & 1^+ & 1^- & 2    \\ \hline  -2 & -2 & -2&4 \end{array} \\ & \\[-3mm] \hline & \\[-3mm]
\textrm{\ Characteristic polynomial} &
 (x-4)(x+2)x^2\\[2mm] \hline & \\[-3mm]
\begin{array}{l}
\textrm{Eigenvalues} \\ 
\textrm{Weights of left eigenvectors}
\end{array} &
\begin{array}{ c |c |c }  4 & -2 & 0 (\times 2) \\ \hline  0 & -1 &   \phantom{}^{**} \end{array}
\\  & \\[-3mm]
\hline \hline
\end{array}
$$

In both Examples~\ref{ex-3-1ue} and \ref{ex-3-2ue}, the rotated odometers are stationary and covering, and so all our theorems apply. The weights of all substitution words are even, thus Theorem~\ref{thm-1.0-non-essential} gives $E(\Psi) \subset 2\mathbb{Z}$ in both cases. By Theorem~\ref{thm-uniqergod-pisot} we conclude that $\overline{E}(\Psi) = \{0,\infty\}$.

}
\end{example}

\begin{example}\label{ex-9-2ue}
{\rm
$$
\begin{array}{l|l}
\hline\hline & \\[-3mm]
\begin{array}{l}
 \pi = (1,7,4)(2,5)(3,6)(0)(8)   \\[2mm]
q = 9 \\[2mm]
\textrm{covering}\\[2mm]
\#\{\textrm{ergodic measures}\} = 1\\[2mm]
 \gamma_0 = \frac12 \\[2mm]
\overline{E}(\Psi) = \{ 0, \infty\} 
\end{array} &     M =  \left(\begin{array}{rrrrrrrrrr}
4 & 3 & 1 & 2 & 0 & 1 & 3 & 4 & 1 & 3 \\
3 & 6 & 3 & 3 & 1 & 1 & 3 & 3 & 2 & 3 \\
1 & 2 & 3 & 3 & 1 & 1 & 1 & 1 & 2 & 1 \\
1 & 2 & 3 & 3 & 1 & 1 & 1 & 1 & 2 & 1 \\
1 & 0 & 1 & 1 & 1 & 1 & 1 & 1 & 2 & 1 \\
1 & 0 & 1 & 1 & 1 & 1 & 1 & 1 & 2 & 1 \\
1 & 0 & 1 & 1 & 1 & 1 & 1 & 1 & 2 & 1 \\
1 & 0 & 1 & 1 & 1 & 1 & 1 & 1 & 2 & 1 \\
1 & 0 & 1 & 1 & 1 & 1 & 1 & 1 & 2 & 1 \\
3 & 3 & 2 & 1 & 1 & 0 & 4 & 3 & 1 & 4
\end{array}\right)
 \\ & \\
 \hline & \\[-3mm]
\begin{array}{l}
a \in \mathcal{A} \\ \textrm{Weight } \psi (\chi (a))
\end{array} & \begin{array}{c |c |c |c |c |c| c| c| c| c} 0 & 1 & 2 & 3 & 4^+ & 4^-& 5 & 6 & 7 & 8   \\ \hline  -2 & 4 & 4 & 4 & -2 & -2 & -2 & -2 & -2 & -2\end{array} \\ & \\[-3mm] \hline & \\[-3mm]
\textrm{\ Characteristic polynomial} &
  (x-16)(x - 4)^2(x-1)^2x^5\\[2mm] \hline & \\[-3mm]
\begin{array}{l}
\textrm{Eigenvalues} \\ 
\textrm{Weights of left eigenvectors}
\end{array} &
\begin{array}{ c |c |c| c }  16 & 4 (\times 2) & 1 (\times 2) & 0 (\times 5) \\ \hline  0 & -\frac{1}{2}, 1 &  0,0 & \phantom{}^{**} \end{array}
\\  & \\[-3mm]

\hline \hline
\end{array}
$$
In this example, the rotated odometer is stationary and covering, and so all our theorems apply. Since the weights of all substitution words are even, by Theorem~\ref{thm-1.0-non-essential} we have $E(\Psi) \subset 2\mathbb{Z}$.  The second largest eigenvalue $\lambda_1 = 4$ has both the algebraic and the geometric multiplicity equal to $2$. Thus Theorem~\ref{thm-uniqergod-pisot} applies, and we conclude that $\overline{E}(\Psi) = \{0,\infty\}$.
}
\end{example}

\subsection{A coboundary}

\begin{example}\label{ex-5-1nue}
{\rm
$$
\begin{array}{l|l}
\hline\hline & \\[-3mm]
\begin{array}{l}
 \pi = (01234), \quad q = 5 \\[2mm]
\textrm{covering}\\[2mm]
\#\{\textrm{ergodic measures}\} = 2 \\[2mm]
\psi \textrm{ is coboundary: }\gamma_0 = 0  \\[2mm]
\overline{E}(\Psi) = \{ 0 \}

\end{array} &  M = \begin{pmatrix}
   1 & 0 & 0 & 0 & 1 & 0 \\
   1 & 0 & 0 & 0 & 1 & 0 \\
   1 & 0 & 0 & 0 & 1 & 0 \\
   1 & 0 & 0 & 0 & 1 & 0 \\
   1 & 0 & 0 & 0 & 1 & 0 \\
   4 & 8 & 4 & 4 & 4 & 8
  \end{pmatrix} \\ & \\

 \hline & \\[-3mm]
\begin{array}{l}
a \in \mathcal{A} \\ \textrm{Weight } \psi (\chi (a))
\end{array} & \begin{array}{c |c |c |c |c |c} 0 & 1 & 2^+ & 2^- & 3 & 4    \\ \hline  0 & 0 & 0&0&0& 0\end{array} \\ & \\[-3mm] \hline & \\[-3mm]
\textrm{\ Characteristic polynomial} &
  (x-8)(x - 2)x^4\\[2mm] \hline & \\[-3mm]
\begin{array}{l}
\textrm{Eigenvalues} \\ 
\textrm{Weights of left eigenvectors}
\end{array} &
\begin{array}{ c |c |c }  8 & 2 & 0 (\times 4) \\ \hline  0 & 0 &   \phantom{}^{**} \end{array}
\\  & \\[-3mm]

\hline \hline
\end{array}
$$

}
\end{example}
For $\pi = (01234)$, the cocycle $\Psi$ is a coboundary because the weight of all substitution words is $0$. Hence all orbits in $(I \times \mathbb{Z}, T_\pi, \mbox{Leb} \otimes \nu)$ are bounded. The minimal subset $(I_{min}, F_\pi)$ corresponds to the symbols $\{0,3\}$, and lifts to the suspension over a dyadic odometer in $(I \times \mathbb{Z}, T_\pi,\mbox{Leb} \otimes \nu)$, which has zero measure with respect to $\mbox{Leb} \otimes \nu$.
 Lebesgue itself is the only other ergodic $F_\pi$-measure.
 
 Note that the weight of the left eigenvector associated to the second largest eigenvalue is $0$, which shows again that the orbits of $T_\pi$ stay within a bounded subset (compare with Theorem~\ref{thm-uniqergod-pisot}).
 
 \subsection{Non-uniquely ergodic covering substitutions}
 
 \begin{example}\label{ex-5-3nue}
{\rm
$$
\begin{array}{l|l}
\hline\hline & \\[-3mm]
\begin{array}{l}
 \pi = (02431), \quad q = 5 \\[2mm]
\textrm{covering}\\[2mm]
\#\{\textrm{ergodic measures}\} = 2 \\[2mm]
\gamma_0 = 0.694... \\[2mm]
\overline{E}(\Psi) = \{ 0, \infty\}
\end{array} & M = \begin{pmatrix}
   1 & 1 & 1 & 1 & 0 & 1 \\
   1 & 0 & 1 & 0 & 0 & 1 \\
   2 & 1 & 0 & 1 & 0 & 1 \\
   2 & 1 & 0 & 1 & 0 & 1 \\
   3 & 5 & 2 & 1 & 8 & 5 \\
   1 & 1 & 0 & 1 & 0 & 0
  \end{pmatrix}
 \\ & \\
 \hline & \\[-3mm]
\begin{array}{l}
a \in \mathcal{A} \\ \textrm{Weight } \psi (\chi (a))
\end{array} & \begin{array}{c |c |c |c |c |c} 0 & 1 & 2^+ & 2^- & 3 & 4    \\ \hline  1 & 1 & 1&1&-4& 1\end{array} \\ & \\[-3mm] \hline & \\[-3mm]
\textrm{\ Characteristic polynomial} &
  (x-8)(x + 1)^2 (x - 2 - \sqrt 5)(x -2 + \sqrt 5)x\\[2mm] \hline & \\[-3mm]
\begin{array}{l}
\textrm{Eigenvalues} \\
\textrm{Weights of left eigenvectors}
\end{array} &
\begin{array}{ c |c |c | c|c }  8 & 2 + \sqrt 5 & -1 (\times 2) & 2 - \sqrt 5 & 0 \\ \hline  0 &   \frac{1}{2}( \frac{3}{\sqrt 5} - 1) & 0,0 & - \frac{1}{2}( \frac{3}{\sqrt 5} + 1) &   \phantom{}^{**} \end{array}
\\  & \\[-3mm]
\end{array}
$$

The rotated odometer in this example has two ergodic invariant measures, one supported on the minimal set, another one Lebesgue. The minimal set corresponds to the symbols $\{ 0,1,2,4\}$, all of which have positive weights. Thus the lift of $(I_{\min},F_\pi,\mu)$, where $\mu$ is the measure supported on the minimal set, is the skew-product $(I \times \mathbb{Z}, T_\pi,\mu \otimes \nu)$ which is transient, namely, $\overline{E}(\Psi|_{\mu \otimes \nu}) = \{0,\infty\}$, with $0$ a trivial essential value. This means that the orbits of typical points with respect to $\mu \otimes \nu$ escape to $\infty$ without returning to $0$. Here, as before, $\nu$ is the counting measure on $\mathbb{Z}$. 

The situation is different for $(I,F_\pi,\mbox{Leb})$. Since the rotated odometer is stationary, and Lebesgue measure is ergodic, Atkinson's result in Lemma~\ref{lemma-lebrecurrent} applies, and $0$ is a non-trivial essential value of $(I,F_\pi,\mbox{Leb})$. There are two eigenvalues outside of the unit circle, with one of left eigenvectors having non-zero weight. Thus Theorem~\ref{thm-uniqergod-pisot} gives $\overline{E}(\Psi) = \{0,\infty\}$. This means that the orbits of typical points with respect to $\mbox{Leb} \otimes \nu$ `oscillate' in $I \times \mathbb{Z}$, going further and further away from $0$, but also returning to $0$ infinitely many times.
 
}
\end{example}

\begin{example}\label{ex-9-1ue}
{\rm
$$
\begin{array}{l|l}
\hline\hline & \\[-3mm]
\begin{array}{l}
 \pi = (0,6,5,8,4,7,3)(1)(2)  \\[2mm]
q = 9 \\[2mm]
\textrm{covering}\\[2mm]
\#\{\textrm{ergodic measures}\} = 2\\[2mm]
 \gamma_0 = 0.5515... \\[2mm]
\overline{E}(\Psi) = \{0, \infty\} \cup \, ?
\end{array} &    M =  \left(\begin{array}{rrrrrrrrrr}
2 & 1 & 0 & 1 & 0 & 0 & 2 & 0 & 0 & 0 \\
1 & 1 & 0 & 1 & 0 & 0 & 2 & 0 & 0 & 0 \\
1 & 1 & 0 & 1 & 0 & 0 & 2 & 0 & 0 & 0 \\
1 & 0 & 0 & 1 & 0 & 0 & 1 & 0 & 0 & 0 \\
1 & 1 & 0 & 0 & 1 & 0 & 1 & 0 & 0 & 0 \\
1 & 1 & 0 & 0 & 1 & 0 & 1 & 0 & 0 & 0 \\
1 & 1 & 0 & 0 & 1 & 0 & 1 & 0 & 0 & 0 \\
5 & 7 & 10 & 10 & 4 & 8 & 3 & 12 & 8 & 8 \\
2 & 2 & 3 & 1 & 1 & 0 & 2 & 2 & 4 & 4 \\
2 & 2 & 3 & 1 & 1 & 0 & 2 & 2 & 4 & 4
\end{array}\right)

 \\ & \\

 \hline & \\[-3mm]
\begin{array}{l}
a \in \mathcal{A} \\ \textrm{Weight } \psi (\chi (a))
\end{array} & \begin{array}{c |c |c |c |c |c| c| c| c| c} 0 & 1 & 2 & 3 & 4^+ & 4^-& 5 & 6 & 7 & 8   \\ \hline  2 & 1 & 1 & 1 & 2 & 2 & 2 & -3 & -3 & -3\end{array} \\ & \\[-3mm] \hline & \\[-3mm]
\textrm{\ Characteristic polynomial} &
  (x-16)(x - 4)(x-1)(x^3 - 5x^2 + 2x -1)x^4\\[2mm] \hline & \\[-3mm]
\begin{array}{l}
\textrm{Eigenvalues} \\ 
\textrm{Weights of left eigenvectors}
\end{array} &
\begin{array}{ c |c |c| c|c|c|c }  16 & 4.614^* & 4 & 1 &0.19 - 0.42i^*  & 0.19 + 0.42i^*   & 0 (\times 4) \\ \hline  0 & 0.48^* &  1 & 0 & 0.26 - 0.36i^* & 0.26+0.36i^* & \phantom{}^{**} \end{array}
\\  & \\[-3mm]

\hline \hline
\end{array}
$$

The rotated odometer in this example has two ergodic invariant measures, as we can see from the Frobenius form of $M$.

For the rotated odometer $(I,F_\pi,\mbox{Leb})$,
 $0$ is a non-trivial essential value by Lemma~\ref{lemma-lebrecurrent}, so Lebesgue typical points with respect to $\mbox{Leb} \otimes \nu$ return infinitely often to $0$. The substitution matrix $M$ has three eigenvalues outside of the unit circle, two integer and one irrational (Pisot) eigenvalue. Left eigenvectors corresponding to two of these eigenvalues have non-zero weights, so Theorem~\ref{thm-uniqergod-pisot} does not apply, but we still conjecture the absence of non-zero integer essential values.
}
\end{example}

\begin{example}\label{ex-q11}
{\rm
$$
\begin{array}{l|l}
\hline\hline & \\[-3mm]
& \\
\begin{array}{l}
\pi = (0,2,7,6,5,4,3,8,10,1,9) \\[2mm]
q = 11 \\[2mm]
\textrm{covering}\\[2mm]
\#\{\textrm{ergodic measures}\} = 2\\[2mm]
\gamma_0 = 0.625... \\[2mm]
\overline{E}(\Psi) \subseteq 2\mathbb{Z} \cup \{\infty\}
\end{array} &
M = \left(\begin{array}{rrrrrrrrrrrr}
2 & 1 & 0 & 1 & 0 & 0 & 0 & 0 & 1 & 1 & 0 & 0 \\
1 & 2 & 1 & 1 & 1 & 0 & 0 & 0 & 1 & 1 & 0 & 0 \\
1 & 0 & 1 & 1 & 1 & 0 & 0 & 0 & 1 & 1 & 0 & 0 \\
1 & 0 & 1 & 1 & 1 & 0 & 0 & 0 & 1 & 1 & 0 & 0 \\
1 & 0 & 0 & 0 & 1 & 0 & 0 & 0 & 1 & 1 & 0 & 0 \\
1 & 0 & 0 & 0 & 1 & 0 & 0 & 0 & 1 & 1 & 0 & 0 \\
1 & 0 & 0 & 0 & 1 & 0 & 0 & 0 & 1 & 1 & 0 & 0 \\
1 & 0 & 0 & 0 & 1 & 0 & 0 & 0 & 1 & 1 & 0 & 0 \\
1 & 0 & 0 & 0 & 1 & 0 & 0 & 0 & 1 & 1 & 0 & 0 \\
1 & 1 & 1 & 0 & 1 & 0 & 0 & 0 & 0 & 2 & 0 & 0 \\
3 & 6 & 6 & 6 & 4 & 4 & 4 & 8 & 4 & 3 & 8 & 8 \\
3 & 6 & 6 & 6 & 4 & 4 & 4 & 8 & 4 & 3 & 8 & 8
\end{array}\right) \\ & \\

 \hline & \\[-3mm]
\begin{array}{l}
a \in \mathcal{A} \\ \textrm{Weight } \psi (\chi (a))
\end{array} & \begin{array}{c |c |c |c |c |c| c| c| c| c| c| c} 0 & 1 & 2 & 3 & 4 & 5^+& 5^- & 6 & 7 & 8 & 9 & 10  \\ \hline  2 & 4 & 2 & 2 & 0 & 0 & 0 & 0 & 0 & 2 & -6 & -6\end{array} \\ & \\[-3mm] \hline & \\[-3mm]
\textrm{\ Characteristic polynomial} &
  (x-16)(x^3 - 8x^2 + 16x - 16)(x-1)^2x^6\\[2mm] \hline & \\[-3mm]
\begin{array}{l}
\textrm{Eigenvalue } \lambda_j \\
\textrm{Weight } \psi(\ell_j)
\end{array} &
\begin{array}{ c |c |c| c| c| c }  16 & 5.679^* & 1.161 - 1.213i^*  & 1.161 + 1.213i^* & 1 (\times 2)& 0 (\times 6) \\ \hline  0 & 0.34^* &  -0.17 + 0.20i^* &
-0.17 - 0.20i^*  & 0, 0 & \phantom{}^{**} \end{array}
\\  & \\[-3mm]
\hline \hline
\end{array}
$$
}
\end{example}

The rotated odometer in this example has two ergodic invariant measures, with one of the measures, $\mu$, supported on the minimal set $I_{min} \subset I$, corresponding to the part of the matrix with symbols $\{0,1,2,3,4, 7, 8\}$. The weights of the substitution words corresponding to these symbols are all positive or zero, thus the skew-product $(I \times \mathbb{Z}, T_\pi,\mu \otimes \nu)$ is transient, namely, $\overline{E}(\Psi|_{\mu \otimes \nu}) = \{0,\infty\}$, with $0$ a trivial essential value. The orbits of typical points with respect to $\mu \otimes \nu$ tend to $\infty$ without returning to $0$. 

For the rotated odometer $(I ,F_\pi,\mbox{Leb})$,  $0$ is a non-trivial essential value
by Lemma~\ref{lemma-lebrecurrent}, so $\mbox{Leb} \otimes \nu$-typical points return infinitely often to $0$. Non-zero weights of substitution words are even, which implies by Theorem~\ref{thm-1.0-non-essential} that $\overline{E}(\Psi) \subset 2\mathbb{Z} \cup \{\infty\}$, so $1$ is not an essential value, and $\mbox{Leb} \otimes \nu$ is not ergodic. The second largest eigenvalue is not Pisot, so Theorem~\ref{thm-uniqergod-pisot} does not apply, and we cannot rule out the existence of non-zero integer essential values.

\begin{example}
$$
\begin{array}{l|l}
\hline\hline & \\[-3mm]
& \\
\begin{array}{l}
\pi =  (0,8,7,3,2,1)(4,9)(5,10,6) \\[2mm]
q = 11 \\[2mm]
\textrm{covering}\\[2mm]
\#\{\textrm{ergodic measures}\} = 3\\[2mm]
\gamma_0 = 0.694 ... \\[2mm]
\overline{E}(\Psi) =  \{0\} \cup \, ?
\end{array} &
M = \left(\begin{array}{rrrrrrrrrrrr}
1 & 1 & 0 & 0 & 0 & 1 & 1 & 0 & 0 & 0 & 0 & 0 \\
1 & 0 & 0 & 0 & 0 & 1 & 0 & 0 & 0 & 0 & 0 & 0 \\
1 & 0 & 0 & 0 & 0 & 1 & 0 & 0 & 0 & 0 & 0 & 0 \\
1 & 0 & 0 & 0 & 0 & 1 & 0 & 0 & 0 & 0 & 0 & 0 \\
1 & 1 & 0 & 0 & 0 & 0 & 1 & 0 & 0 & 0 & 0 & 0 \\
1 & 1 & 0 & 0 & 0 & 0 & 1 & 0 & 0 & 0 & 0 & 0 \\
1 & 1 & 0 & 0 & 0 & 0 & 1 & 0 & 0 & 0 & 0 & 0 \\
2 & 2 & 3 & 1 & 2 & 1 & 1 & 2 & 2 & 3 & 0 & 1 \\
1 & 2 & 1 & 1 & 2 & 0 & 1 & 1 & 1 & 2 & 0 & 1 \\
2 & 3 & 2 & 1 & 2 & 1 & 1 & 3 & 1 & 2 & 0 & 2 \\
4 & 4 & 8 & 13 & 9 & 1 & 1 & 8 & 11 & 8 & 16 & 11 \\
1 & 2 & 2 & 0 & 1 & 1 & 1 & 2 & 1 & 1 & 0 & 1
\end{array}\right) \\ & \\

 \hline & \\[-3mm]
\begin{array}{l}
a \in \mathcal{A} \\ \textrm{Weight } \psi (\chi (a))
\end{array} & \begin{array}{c |c |c |c |c |c| c| c| c| c| c| c} 0 & 1 & 2 & 3 & 4 & 5^+& 5^- & 6 & 7 & 8 & 9 & 10  \\ \hline  2 & 2 & 2 & 2 & 1 & 1 & 1 & 2 & 1 & 2& -16 &1  \end{array} \\ & \\[-3mm] \hline & \\[-3mm]
\textrm{\ Characteristic polynomial} &
  (x-16)(x^2 - 7x+1)(x^2+x+1)(x^3 - 2x^2-3x-1)x^4\\[2mm] \hline & \\[-3mm]
\begin{array}{l}
\textrm{Eigenvalue } \lambda_j \\
\textrm{Weight } \psi(\ell_j)
\end{array} &
\begin{array}{ c |c |c| c| c| c |c  }  16 & 6.85^* & 3.08^*  & 0.14^* & -0.5 \pm 0.86 i^* & - 0.54 \pm 0.18i& 0 (\times 4) \\ \hline  0 & 1.4^* & 1.6^* & 3.6^*&  0,0 &
1.2 \mp 1.3i^*   & \phantom{}^{**} \end{array}
\\  & \\[-3mm]

\hline \hline
\end{array}
$$

The rotated odometer in this example has three ergodic invariant measures, which can be seen from the Frobenius form of $M$:
$$ 
\left(\begin{array}{rrrr|rrr|rrrr|r}
1 & 1 & 1 & 1 & 0 & 0 & 0 & 0 & 0 & 0 & 0 & 0 \\
1 & 0 & 0 & 1 & 0 & 0 & 0 & 0 & 0 & 0 & 0 & 0 \\
1 & 1 & 1 & 1 & 0 & 0 & 0 & 0 & 0 & 0 & 0 & 0 \\
1 & 1 & 1 & 1 & 0 & 0 & 0 & 0 & 0 & 0 & 0 & 0 \\ \hline
1 & 1 & 1 & 0 & 0 & 0 & 0 & 0 & 0 & 0 & 0 & 0 \\
1 & 0 & 0 & 0 & 0 & 0 & 0 & 0 & 0 & 0 & 0 & 0 \\
1 & 0 & 0 & 0 & 0 & 0 & 0 & 0 & 0 & 0 & 0 & 0 \\ \hline
2 & 2 & 1 & 1 & 2 & 1 & 3 & 2 & 2 & 3 & 1 & 0 \\
1 & 2 & 1 & 1 & 2 & 1 & 1 & 1 & 1 & 2 & 1 & 0 \\
2 & 3 & 1 & 1 & 2 & 1 & 2 & 3 & 1 & 2 & 2 & 0 \\
1 & 2 & 1 & 0 & 1 & 0 & 2 & 2 & 1 & 1 & 1 & 0 \\ \hline
4 & 4 & 1 & 1 & 9 & 13 & 8 & 8 & 11 & 8 & 11 & 16 
\end{array}\right)
$$
Denote by $\mu_1$ and $\mu_2$ the measures corresponding to the first and the second non-zero block, with leading eigenvalues $3.08... < 6.85...$, respectively. The weights of the symbols in these blocks are all positive, which means that $(I \times \mathbb{Z},T_\pi, \mu_1 \otimes \nu)$ and $(I\times \mathbb{Z},T_\pi, \mu_2 \otimes \nu)$ are both transient. Since the rotated odometer is covering, $(I \times \mathbb{Z},T_\pi, \mbox{Leb} \otimes \nu)$ is recurrent by Lemma~\ref{lemma-lebrecurrent}. However, $\gcd\{\psi(\chi(a)) : a \in \mathcal{A}\} = 1$ so Theorem~\ref{thm-1.0-non-essential} does not allow us to rule out any integer essential values. Since $M$ has multiple eigenvalues with non-zero weights of left eigenvectors outside of the unit circle, Theorem~\ref{thm-uniqergod-pisot} does not apply, and we cannot decide if $\mbox{Leb} \otimes \nu$ is ergodic or not.

\end{example}

 \subsection{Subsitutions which are not covering}
 
 \begin{example}\label{ex-7-1ncue}
{\rm
$$
\begin{array}{l|l}
\hline\hline & \\[-3mm]
\begin{array}{l}
 \pi = (0,6,5,4,3,2,1) \\[2mm]
q = 7 \\[2mm]
\textrm{not covering}\\[2mm]
\#\{\textrm{ergodic measures}\} = \infty \\[2mm]
\textrm{transient: } \gamma_0 = 1 \\[2mm]
\overline{E}(\Psi) = \{ \infty\}

\end{array} &   M = \begin{pmatrix}
2 & 2 & 0 & 0 & 1 & 1 & 0 & 2 \\
1 & 0 & 0 & 0 & 0 & 0 & 0 & 0 \\
1 & 0 & 0 & 0 & 0 & 0 & 0 & 0 \\
1 & 0 & 0 & 0 & 0 & 0 & 0 & 0 \\
1 & 0 & 0 & 0 & 0 & 0 & 0 & 0 \\
1 & 0 & 0 & 0 & 0 & 0 & 0 & 0 \\
1 & 0 & 0 & 0 & 0 & 0 & 0 & 0 \\
1 & 0 & 0 & 0 & 0 & 0 & 0 & 0
  \end{pmatrix}
 \\ & \\

 \hline & \\[-3mm]
\begin{array}{l}
a \in \mathcal{A} \\ \textrm{Weight } \psi (\chi (a))
\end{array} & \begin{array}{c |c |c |c |c |c| c|c} 0 & 1 & 2 & 3^+ & 3^- & 4 & 5 & 6    \\ \hline  0 & 1 & 1 & 1 & 1 & 1 & 1 & 1 \end{array} \\ & \\[-3mm] \hline & \\[-3mm]
\textrm{\ Characteristic polynomial} &
(x^2-2x-6) x^6\\[2mm] \hline & \\[-3mm]
\begin{array}{l}
\textrm{Eigenvalues} \\ 
\textrm{Weights of left eigenvectors}
\end{array} &
\begin{array}{ c |c |c }  1 + \sqrt{7} & 1- \sqrt 7  & 0 (\times 4) \\ \hline  \frac{1}{3}(4- \sqrt 7)& \frac{1}{3}(4 + \sqrt 7) & \phantom{}^{**} \end{array}
\\  & \\[-3mm]

\hline \hline
\end{array}
$$

The rotated odometer in this example is not covering, which verified by noticing that the column sums in the substitution matrix are not $2^N = 8$, and which means  that $I = I_{per} \cup I_{np}$, where $I_{per}$ is a non-empty union of intervals of periodic points. Since the substitution is stationary, $I_{per}$ is an infinite union of intervals, with infinite number of periods. 
Since $\psi(\chi(a)) = 1 > \psi(\chi(0)) = 0$ for all $a \ne 0 \in \mathcal{A}$,
the $z$-extension $(I_{np} \times \mathbb{Z}, T_\pi)$ is transient.

}
\end{example}

 \begin{example}\label{ex-5-2nue}
{\rm
$$
\begin{array}{l|l}
\hline\hline & \\[-3mm]
\begin{array}{l}
 \pi = (04123), \quad q = 5 \\[2mm]
\textrm{not covering}\\[2mm]
\#\{\textrm{ergodic measures}\} = 2 
\end{array} & M = \begin{pmatrix}
   2 & 0 & 0 & 0 & 0 & 0 \\
   2 & 3 & 2 & 1 & 4 & 4 \\
   2 & 3 & 2 & 1 & 4 & 4 \\
   2 & 3 & 2 & 1 & 4 & 4 \\
   1 & 0 & 0 & 0 & 0 & 0 \\
   1 & 0 & 0 & 0 & 0 & 0
  \end{pmatrix}
 \\
 \hline & \\[-3mm]
\begin{array}{l}
a \in \mathcal{A} \\ \textrm{Weight } \psi (\chi (a))
\end{array} & \begin{array}{c |c |c |c |c |c} 0 & 1 & 2^+ & 2^- & 3 & 4    \\ \hline  2 & -2 & -2&-2&1& 1\end{array} \\ & \\[-3mm] \hline & \\[-3mm]
\textrm{\ Characteristic polynomial} &
  (x-6)(x - 2)x^4\\[2mm] \hline & \\[-3mm]
\begin{array}{l}
\textrm{Eigenvalues} \\ 
\textrm{Weights of left eigenvectors}
\end{array} &
\begin{array}{ c |c |c }  6 & 2 & 0 (\times 4) \\ \hline  \frac{1}{6} & \frac{1}{2} &   \phantom{}^{**} \end{array}
\\  & \\[-3mm]
\hline \hline
\end{array}
$$
The weight of the leading left eigenvector is positive, which is possible since this substitution is not covering. The minimal part is a dyadic odometer; it lifts to a subsystem on the staircase that is also transient to $+\infty$, with diffusion coefficient $1$. }
\end{example}

\subsection{A non-stationary example}

\begin{example}\label{ex-5-3nue-ns}
{\rm
$$
\begin{array}{l|l}
\hline\hline & \\[-3mm]
\begin{array}{l}
 \pi = (01243), \quad q = 5 \\[2mm]
\textrm{covering}\\[2mm]
\#\{\textrm{ergodic measures}\} = 2 \\[2mm]
\gamma_0 = \frac13, \quad \overline{E}(\Psi) = \{ 0, \infty\}

\end{array} & M = \begin{pmatrix}
   2 & 0 & 0 & 0 & 0 & 0 \\
   2 & 4 & 2 & 2 & 4 & 4 \\
   2 & 4 & 2 & 2 & 4 & 4 \\
   2 & 4 & 2 & 2 & 4 & 4 \\
   1 & 0 & 0 & 0 & 0 & 0 \\
   1 & 0 & 0 & 0 & 0 & 0
  \end{pmatrix}
 \\ & \\

 \hline & \\[-3mm]
\begin{array}{l}
a \in \mathcal{A} \\ \textrm{Weight } \psi (\chi (a))
\end{array} & \begin{array}{c |c |c |c |c |c} 0 & 1 & 2^+ & 2^- & 3 & 4    \\ \hline  2 & -2 & -2&-2& 1& 1\end{array} \\ & \\[-3mm] \hline & \\[-3mm]
\textrm{\ Characteristic polynomial} &
  (x-8)(x -2) x^4\\[2mm] \hline & \\[-3mm]
\begin{array}{l}
\textrm{Eigenvalues} \\ 
\textrm{Weights of left eigenvectors}
\end{array} &
\begin{array}{ c |c |c  }  8 & 2 & 0 (\times 4) \\ \hline  0 & -2 &   \phantom{}^{**} \end{array}
\\  & \\[-3mm]
\hline \hline
\end{array}
$$
For $\pi = (01243)$, the renormalization of $R_{\pi}$ gives $R_{\pi'}$ for $\pi' = (01234)$, see Example~\ref{ex-5-1nue}.  In that example, $\Psi$ is a coboundary, so every orbit of the $\mathbb{Z}$-extension is bounded. 
The composition with $F_\pi$ changes this. For example, the minimal set of $F_{\pi'}$ in Example~\ref{ex-5-1nue} is the dyadic odometer, the minimal subsystem corresponds to the subsystem $\chi'(0) = \chi'(3) = 03$. Because $\psi(\chi(0)) = 2$ and $\psi(\chi(3)) = 1$ for $F_\pi$,
the $\mathbb{Z}$-extension of the minimal subsystem of $F_{\pi}$, which is also the dyadic odometer, becomes transient.

}
\end{example}

\end{document}